\documentclass[10pt]{amsart}
\usepackage{amsmath}
\usepackage{amssymb, amsfonts, amsmath, amsthm, amscd}
\usepackage{mathrsfs}
\usepackage{color, graphics}
\usepackage[all]{xy}
\usepackage{hyperref}
\usepackage{amsxtra}

\newtheorem{theorem}{Theorem}[section]
\newtheorem{proposition}[theorem]{Proposition}
\newtheorem{lemma}[theorem]{Lemma}
\newtheorem{claim}[theorem]{Claim}
\newtheorem{corollary}[theorem]{Corollary}

\newtheorem{remark}[theorem]{Remark}

\theoremstyle{remark}

\numberwithin{equation}{section}

\def\deg{\operatorname{deg}}%
\def\dim{\operatorname{dim}}%
\def\max{\operatorname{max}}%
\def\ker{\operatorname{ker}}%

\def\pic{\hbox{\rm Pic}}

\def\deg{\operatorname{deg}}%
\def\dim{\operatorname{dim}}%
\def\max{\operatorname{max}}%
\def\ker{\operatorname{ker}}%

\newcommand \im   {\ensuremath{\mathrm{im}}}

\newcommand \ext {\ensuremath{\mathrm{Ext}}}
\newcommand \Sec {\ensuremath{\mathrm{Sec}}}

\newcommand \coker {\ensuremath{\mathrm{coker}}}

\def\deg{\mbox{deg}}

\def\Pp{\mathbb P}

\def\Ff{\mathcal F}
\def\N{\mathcal N}
\def\T{\mathcal T}
\def\OO{\mathcal O}
\def\Oc{\mathcal O}

\begin{document}

\title[Moduli spaces of bundles and Hilbert schemes of scrolls
over $\nu$-gonal curves]{Moduli spaces of bundles and Hilbert schemes of scrolls
over $\nu$-gonal curves}
\author{Youngook Choi}
\address{Department of Mathematics Education, Yeungnam University, 280 Daehak-Ro, Gyeongsan, Gyeongbuk 38541,
 Republic of Korea }
\email{ychoi824@yu.ac.kr}
\author{Flaminio Flamini}
\address{Universita' degli Studi di Roma Tor Vergata, 
Dipartimento di Matematica, Via della Ricerca Scientifica-00133 Roma, Italy}
\email{flamini@mat.uniroma2.it}
\author{Seonja Kim}
\address{Department of  Electronic Engineering,
Chungwoon University, Sukgol-ro, Nam-gu, Incheon, 22100, Republic of Korea}
\email{sjkim@chungwoon.ac.kr}
\thanks{The first author was supported by Basic Science Research Program through the National Research Foundation of Korea (NRF-2016R1D1A3B03933342) and by Italian PRIN $2015EYPTSB-011$-{\em Geomety of Algebraic varieties} (Node Tor Vergata).  The third  author was supported by Basic Science Research Program through the National Research Foundation of Korea (NRF-2016R1D1A1B03930844) and by Italian PRIN $2015EYPTSB-011$-{\em Geomety of Algebraic varieties} (Node Tor Vergata). For their collaboration, the three authors have been supported by funds {\em Mission Sustainability 2017 - Fam Curves}. CUP E81|18000100005 (Tor Vergata University)}


\begin{abstract} The aim of this paper is two--fold. We first strongly improve our previous main result \cite[Theorem\;3.1]{CFK}, concerning classification of irreducible components of the Brill--Noether locus 
parametrizing rank 2 semistable vector bundles of suitable degrees $d$, with at least $d-2g+4$ independent global sections, on a general $\nu$--gonal curve $C$ of genus $g$. 
We then uses this classification to study several properties of the Hilbert scheme of suitable surface scrolls in projective space, which turn out to be special and stable.  
\end{abstract}

\maketitle



\section{Introduction}
Let $C$ denote a smooth, irreducible, complex projective curve of genus $g \geq 3$. Let $U_C(2, d)$ be the moduli space of semistable, degree $d$, rank $2$ vector bundles on $C$ and let $U^s_C(2,d)$ be the open dense  subset of stable bundles (when $d$ is odd, more precisely one has $U_C(2, d)=U^s_C(2,d)$). Let $B^k_{2,d} \subseteq U_C(2, d)$ be the {\em Brill-Noether locus} which consists of vector bundles 
$\mathcal F$  having $h^0(\mathcal F)\ge k$, for a positive integer $k$. 

Traditionally, one denotes by $W^k_d$ the Brill-Noether locus  $B^{k+1}_{1,d}$ of line bundles
$L\in \mbox{Pic}^d(C)$ having $h^0(L)\ge k+1$, for a non-negative integer $k$. In what follows, we sometimes identify line 
bundles with corresponding divisor classes, interchangeably using multiplicative and additive notation.

For the case of rank $2$ vector bundles, we simply put  $B^k_d:=B^k_{2,d}$, 
for which it is well-known that the {\em expected dimension} of $B_d^k \cap U^s_C(2,d)$ is $\rho_d^{k}:=4g-3-ik$, where $i:=k+2g-2-d$ (cf.\;\cite{Sun}). 
Recall that, as customary, an irreducible component of $B_d^k$ is said to be {\em regular}, if it is reduced with expected dimension, and {\em superabundant}, otherwise.

In the range $0 \le d \le 2g-2$, $B^1_d$ has been deeply studied on any curve $C$ by several authors (cf.\;e.g.\;\cite{Sun,L}).
Concerning $B^2_d$,  using a degeneration argument, N. Sundaram \cite{Sun} proved  that $B^2_d$ is non-empty for any  $C$ and for odd $d$ such that $g\le d\le 2g-3$. 
M. Teixidor I Bigas generalizes Sundaram's result as follows: 

\begin{theorem}[\cite{Teixidor1}]\label{Teixidor} Given a non-singular curve $C$ of genus $g$ and an integer $d$, where \linebreak $3\le d\le 2g-1$, then 
$B^2_d \cap U^s_C(2,d)$ has a component $\mathcal B$ of (expected)  dimension $\rho^2_d=2d-3$ and a general point on it corresponds to a vector bundle whose space of sections has dimension $2$. 
If $C$ is general (i.e. $C$ is a curve {\em with general moduli}), this is the only component of $B^2_d\cap U^s_C(2,d)$. Moreover, $B^2_d\cap U^s_C(2,d)$ has extra components if and only 
if $W^1_n$ is non-empty, with $\dim  W^1_n\ge d+2n-2g-1$, for some integer $n$ such that $2n<d$.
\end{theorem} 

\noindent
\begin{remark}\label{rem:TeixRes} {\normalfont The previous result is sharp concerning non-emptiness of $B^2_d \cap U^s_C(2,d)$; indeed, on any smooth curve $C$ of genus $g \geq 3$ one has  
$B^2_d \cap U^s_C(2,d) = \emptyset$ for $d = 0,\; 1,\; 2$ (cf. \cite{Teixidor1}). Moreover, Theorem \ref{Teixidor} has a {\em residual version}, giving information also on the isomorphic Brill Noether locus \linebreak $B^{d-2g+4}_{4g-4-d}\cap U^s_C(2,4g-4-d)$. 
Indeed, for any non--negative integer $i$, if one sets $k_i := d-2g+2+i$ and $$B_d^{k_i} : = \{\mathcal F \in U_C(2,d)\;|\: h^0(\mathcal F) \ge k_i\} =  \{\mathcal F \in U_C(2,d)\;|\: h^1(\mathcal F) \ge i\},$$one has natural isomorphisms 
$B_d^{k_i} \simeq B_{4g-4-d}^{i}$, which arise from the natural correspondence \linebreak $\mathcal F \to \omega_C \otimes \mathcal F^*$, from Serre duality and from semistability. Under this  natural {\em residual correspondence} one has: 
}
\end{remark}

\begin{theorem}[Residual Version of Theorem \ref{Teixidor}]\label{TeixidorRes} Given a non-singular curve $C$ of genus $g$, an integer $d$, where $2g-3\le d\le 4g-7$, let $k_2:= d-2g+4$. Then,   
$B^{k_2}_d \cap U^s_C(2,d)$ has a component $\mathcal B$ of (expected)  dimension $\rho^{k_2}_d=8g-2d-11$ and a general point on it corresponds to a vector bundle whose space of sections has dimension $k_2$. If $C$ is general, this is the only 
component of $B^{k_2}_d \cap U^s_C(2,d)$. Moreover, $B^{k_2}_d \cap U^s_C(2,d)$ has extra components if and only if $W^1_n$ 
is non-empty with $\dim  W^1_n\ge 2g + 2n - d - 5$,  for some integer $n$ such that $2n<4g-4-d$.
\end{theorem} Inspired by Theorem \ref{TeixidorRes}, in \cite{CFK} we focused on $B^{k_2}_d$ as above, for $C$ a {\em general $\nu$-gonal curve} of genus $g$, i.e. $C$ corresponding to a general point of the 
{\em $\nu$-gonal stratum} $\mathcal  M^1_{g,\nu} \subset \mathcal M_g$. Observe that in this case, as a consequence of Theorem \ref{TeixidorRes}, $B^{k_2}_d \cap U^s_C(2,d)$ is empty for 
$d = 4g-4,\; 4g-5, \;4g-6$ and it consists only of the irreducible component  $\mathcal B$ as in Theorem \ref{TeixidorRes}, for any $4g-4-2\nu \leq d \leq 4g-7$ (cf.\;Remark\;\ref{rem:important} below).

Concerning the residual values for $d$, the aim of this paper is two--fold. The first is to strongly improve \cite[Theorem\;3.1]{CFK}, \color{black} where we proved the following result: 

\begin{theorem}\label{thm3.1}(cf.\;\cite[Theorem\;3.1]{CFK}) Let 
$$3\le \nu \le \frac{g+8}{4} \;\; {\rm and} \;\;3g-1\le d\le 4g-6-2 \nu$$be integers. Then, the reduced components of $B^{k_2}_d \cap U^s_C(2,d)$ are only two,  which we denote by $B_{\rm reg}$ and $B_{\rm sup}$:
\begin{enumerate}
\item[(i)]
 The component $B_{\rm reg}$ is {\em regular},  i.e. generically smooth and of  dimension $\rho^{k_2}_d=8g-2d-11$.    A general element $\mathcal F$ of $B_{\rm reg}$ fits in an exact sequence
 $$ 0\to \omega_C(-D) \to \mathcal F\to \omega_C(-p)\to 0,$$where $p\in C$ and $D\in C^{(4g-5-d)}$ are general. Specifically,  $s_1(\mathcal F) \ge 1$ (resp., $2$)  if $d$ is odd (resp., even). Moreover, $\omega_C(-p)$ is of minimal degree among special quotient line bundles of $\mathcal F$ and $\mathcal F$ is very ample for $\nu \ge 4$;  
\item[(ii)] The component $B_{\rm sup}$ is generically smooth,  of dimension $6g-d-2 \nu -6 > \rho^{k_2}_d$, i.e. $B_{\rm sup}$ is {\em superabundant}. A general element $\mathcal F$ of $B_{\rm sup}$ is  very-ample and   fits in an exact sequence
$$0\to N\to \mathcal F\to \omega_C\otimes A^{\vee} \to 0,$$where $A \in \pic^{\nu}(C)$ such that $|A| = g^1_{\nu}$ on $C$ and 
where $N\in\pic^{d-2g+2+\nu}(C)$ general. Moreover, 
$s_1(\mathcal F)=4g-4 - d - 2\nu$ and $\omega_C \otimes A^{\vee}$ is of minimal degree 
among quotient line bundles of $\mathcal F$.
\end{enumerate}
\end{theorem}

In the present paper, \color{black} under conditions 
\begin{eqnarray}\label{eq:ourbounds}
\nu \geq 3 \;\; {\rm and} \;\;3g-5\le d\le 4g-5-2 \nu, 
\end{eqnarray} we first prove the following:

\begin{theorem}\label{thm:main2} For $\nu$ and $d$ as in \eqref{eq:ourbounds}, the irreducible components of $B^{k_2}_d \color{black} \cap 
U^s_C (2,d)$ \color{black} are only two,  which we denote by $B_{\rm reg}$ and $B_{\rm sup}$. 
\begin{enumerate}
\item[(i)] The component $B_{\rm reg}$ is {\em regular}  and {\em uniruled}.  A general element $\mathcal F$ of $B_{\rm reg}$    fits   in an exact sequence
\begin{equation}\label{exactB0} 
 0\to \color{black} \omega_C(-D) \color{black} \to \mathcal F\to  \color{black} \omega_C(-p) \color{black} \to 0,  
\end{equation}where $p\in C$ and $D\in C^{(4g-5-d)}$ are general. Moreover, $ \color{black} \omega_C(-p) \color{black}$ is \color{black}of minimal degree \color{black} among special quotient line bundles of $\mathcal F$.
  
\item[(ii)] If $3g-3 \leq d \leq 4g-5-2\nu$, the component $B_{\rm sup}$ is generically smooth,  of dimension $6g-d-2 \nu -6 > \rho^{k_2}_d$, i.e. $B_{\rm sup}$ is {\em superabundant}, and {\em ruled}.

\noindent
If otherwise $d = 3g-5,\;3g-4$,  the component $B_{\rm sup}$ is of dimension $6g-d-2 \nu -6 \geq \rho^{k_2}_d$, where equality holds only for $\nu=\frac{g}{2}$ and $d=3g-5$; 
the component $B_{\rm sup}$ is {\em ruled} and {\em superabundant}, being non--reduced. 

In any case, a general element $\mathcal F$ of $B_{\rm sup}$    fits   in an exact sequence
\begin{equation}\label{exactB1}
0\to N\to \mathcal F\to  \color{black} \omega_C \otimes A^{\vee} \to 0, 
\end{equation} \color{black} where $A \in \pic^{\nu}(C)$  such that  $|A| = g^1_{\nu}$ on $C$ and where \color{black}  $N\in\pic^{d-2g+2+\nu}(C)$ is general. Moreover, 
$s(\mathcal F)=4g-4 - d - 2\nu$ and $\color{black} \omega_C \otimes A^{\vee} $ \color{black} is \color{black} of minimal degree 
among quotient line bundles \color{black} of $\mathcal F$.
\end{enumerate}
\end{theorem}

\begin{remark}{\normalfont  (i) Notice that Theorem \ref{thm:main2} strongly improves \cite[Theorem 3.1]{CFK} \color{black} (i.e. Theorem \ref{thm3.1} reported above) \color{black}  from several points of view. 
First of all, Theorem \ref{thm:main2} holds for any $\nu \geq 3$ and for any $3g-5 \leq d \leq 4g-5-2\nu$, whereas \cite[Theorem 3.1]{CFK} was proved under the 
assumptions $3 \leq \nu \leq \frac{g+8}{4} \;\; {\rm and} \;\; 3g-1 \leq d \leq 4g-6-2\nu$ (cf. \cite[formula\;(3.1)]{CFK}). Moreover, no reducedness assumption appears in the statement of Theorem \ref{thm:main2}, as it occurred in \cite[Theorem 3.1]{CFK}. In fact, for $d = 3g-5,\; 3g-4$, in this paper we discover non--reduced components of  $B_d^{k_2} \cap U^s_C(2,d)$. Finally, Theorem \ref{thm:main2} 
contains important information on the {\em birational structure} of these irreducible components (i.e. ruledness, uniruledness, etc.).

\vskip 3pt

\noindent
(ii) Theorem \ref{thm:main2} also exhibits the postulated (by Theorem \ref{TeixidorRes}) reducibility of $B^{k_2}_d\cap U_C^s(2,d)$ for a general $\nu$-gonal curve, 
in the interval $3g-5\le d\le 4g-5-2 \nu$. Indeed,  in such cases, one can take $n = \nu$ and $W^1_{\nu} = \{A\}$.

\vskip 3pt

\noindent
(iii) Theorem \ref{thm:main2} moreover shows that Teixidor I Bigas' component $\mathcal B$ in Theorem \ref{TeixidorRes} coincides with our component $B_{\rm reg}$ (cf. Remark \ref{rem:min} below).

\vskip 3pt

\noindent
(iv) The strategies used in the proof of Theorem \ref{thm:main2} can be used also for the remaining values of $d$, namely 
$2g-3 \leq d \leq 3g-6$. One can find non--emptiness results also in these cases, nevertheless the description of the (birational) geometry of $B^{k_2}_d$ and of a general point of any 
of its irreducible components is not as precise as in the statement of Theorem \ref{thm:main2}.

\vskip 3pt 

\noindent
(v) \color{black} Other consequences \color{black} of Theorem \ref{thm:main2}  are also discussed (cf.\;\color{black}Corollary\color{black}\;\ref{fixeddeterminant} below). 
} 
\end{remark}

The second aim of the paper is concerned with Hilbert schemes of surface scrolls. Precisely, after proving some sufficient very--ampleness conditions for bundles arising from 
Theorem \ref{thm:main2} (cf.\;Theorem \ref{thm:veryampleness} below), we can study suitable components of the Hilbert scheme $\mathcal H_{d,g,k_2-1}$ parametrizing smooth surface scrolls $S$ of degree $d$, 
sectional genus $g$, speciality $2$, which are linearly normal in $\mathbb{P}^{k_2-1}$. 

The range for $d$ in which we study Hilbert schemes are taken from Theorem \ref{thm:veryampleness} and from \eqref{eq:ourbounds}. Notice indeed that the inequality 
$d\leq 4g-5-2\nu$ in \eqref{eq:ourbounds} yields $ d\leq 4g-11$ for $\nu=3$ and hence the study of Hilbert schemes of surface scrolls will be considered in this maximal range $ d\leq 4g-11$.
Precisely, we show:

\begin{theorem}\label{thm:Hilb} Consider the Hilbert scheme $\mathcal H_{d,g,k_2-1}$ as above. Then:  
\vskip 2pt

\noindent
(i)  for $3g-1\leq d \leq 4g-11$, $\mathcal H_{d,g,k_2-1}$ contains an irreducible component, denoted by $\mathcal H_{\rm reg}$,  whose  general point 
corresponds to a smooth scroll $S$ as above, arising from a stable bundle as in \eqref{exactB0} where $C$ is with general moduli 
(i.e. $\mathcal H_{\rm reg}$ dominates $\mathcal M_g$ ). Moreover, $\mathcal H_{\rm reg}$ is generically smooth, of (expected) dimension 
$$\dim\; \mathcal H_{\rm reg} =7g-7 + k_2(k_2-2) = 7g-7 + (d-2g+4)(d-2g+2)$$i.e. it is a {\em regular} component of $\mathcal H_{d,g,k_2-1}$. Scrolls arising from general (very ample) bundles 
in $B_{\rm reg}$  as in Theorem \ref{thm:main2} (i) fill--up a closed subset $\mathcal Y' \subsetneq \mathcal H_{\rm reg}$, 
where $\mathcal Y'$ dominates $\mathcal M^1_{g,\nu}$;

\vskip 2pt

\noindent
(ii)  for $3g-2\leq d\leq 4g-11$ (resp., $d=3g-3$),  $\mathcal H_{d,g,k_2-1}$ carries distinct irreducible components $\mathcal H_{\rm sup, \nu}$,  
for any $\nu$ with $3\leq \nu \leq [\frac{4g-5-d}{2}]$ (resp., $4\leq \nu \leq [\frac{4g-5-d}{2}]$).
$\mathcal H_{\rm sup, \nu}$ is generically smooth of dimension 
$$\dim\; \mathcal H_{\rm sup, \nu} =8g-d-12 + k_2^2 = 8g-d - 12 + (d - 2g + 4)^2$$
which is higher than the expected one, so it is a {\em superabundant} component of $\mathcal H_{d,g,k_2-1}$, unless $d=3g-3$. In case $d=3g-3$, $\mathcal H_{\rm sup, \nu}$ is a regular component for every possible $\nu$.

For any such $d$, $\mathcal H_{\rm sup, \nu}$ dominates $\mathcal M^1_{g,\nu}$ and  its general point corresponds to a smooth scroll $S$ as above, arising from a general  
$\mathcal F \in B_{\rm sup}$ as in Theorem \ref{thm:main2}. 
\end{theorem} 

\noindent
The rest of the paper will be concerned with the proof of the aforementioned results.

\color{black} In what follows, we may sometimes abuse notation and identify divisor classes with the corresponding line bundles, interchangeably using additive and multiplicative notation when this does not create ambiguity. \color{black}
For standard terminology, we refer the reader to \cite{H}.

\smallskip

\noindent
{\bf Acknowledgements}. The authors thank KIAS and Dipartimento di Matematica Universita' di Roma "Tor Vergata" for the warm atmosphere and hospitality during the collaboration and the preparation of this article. \color{black} The authors are indebted to the referee for the careful reading of the first version of the paper and for valuable comments and suggestions which have certainly improved the readability of the paper.\color{black}


\section{Preliminaries}\label{S:Pre} In what follows, $C$ will always denote a smooth, irreducible, projective curve of genus $g \geq 3$.  
We recall some standard notation and results frequently used below. 

Given a rank $2$ vector bundle $\mathcal F$ on $C$, the \textit{Segre invariant} $s(\mathcal F) \in \mathbb{Z}$ of $\mathcal F$ 
is defined by$$s(\mathcal F) = \min_{N \subset \mathcal F} \left\{   \deg\; \mathcal F  -   2\; \deg \; N  \right\},$$
where $N$ runs through all the sub-line bundles of $\mathcal F$.  One has 
$s(\mathcal F) = s(\mathcal F \otimes L)$, for any line bundle $L$, and 
$s(\mathcal F) = s (\mathcal F^*)$, where $\mathcal F^*$ denotes the dual bundle of $\mathcal F$. A sub-line bundle $ N \subset \mathcal F$ is called a \textit{maximal sub-line bundle} of $\mathcal F$ if $\deg \; N$ is maximal among all sub-line bundles of $\mathcal F$; in such a case $\mathcal F/N$ is a \textit{minimal  quotient line bundle} of $\mathcal F$, i.e. is of minimal degree among quotient line bundles of $\mathcal F$. In particular, $\mathcal F$ is {\em semistable}  (resp. {\em stable}) if and only if 
$s(\mathcal F) \ge 0$ (resp.   $s(\mathcal F) > 0$). 

Let $\delta$ be a positive integer. Consider $L\in \pic^\delta(C)$ and $N\in\pic^{d-\delta}(C)$. 
The extension space $\ext^1(L,N)$ parametrizes isomorphism classes of extensions and any vector $u\in\ext^1(L,N)$ gives rise to a degree $d$, rank $2$ vector bundle $\mathcal F_u$, fitting in an exact sequence
\begin{equation}\label{degree}
(u):\;\; 0 \to N \to \mathcal F_u \to L \to 0.
\end{equation}

\noindent 
In order for $\mathcal F_u$ as above to be semistable, a necessary condition is 
\begin{equation}\label{eq:neccond}
 2\delta-d \ge s(\mathcal F_u)\ge 0.
\end{equation} In such a case, the Riemann-Roch theorem gives 
\begin{equation}\label{eq:m}
\dim \; \ext^1(L,N)=
\begin{cases}
\ 2\delta-d+g-1\ &\text{ if } L\ncong N \\
\ g\ &\text{ if } L\cong  N.
\end{cases}
\end{equation}

Since we will deal with {\em special} rank 2 vector bundles $\mathcal F_u$, i.e. $h^1(\mathcal F_u) >0$, then $\mathcal F_u$ always admits a special quotient line bundle. 
Recall the following:

\begin{theorem}\label{CF}(\cite[Lemma\;4.1]{CF})
Let $\mathcal F$ be a semistable, special, rank $2$ vector bundle on $C$ of 
degree $d\ge 2g-2$. Then there exist a special, effective line bundle $L$ on $C$, of degree $\delta \leq d$, $N \in {\rm Pic}^{d-\delta}(C)$ and $ u \in \ext^1(L,N)$ such that 
$\mathcal F = \mathcal F_u$ as in \eqref{degree}.
\end{theorem}

Take $\mathcal F_u$ as in \eqref{degree}. When $(u)$ does not split, it defines a point $[(u)] \in \mathbb P (\ext^1(L,N)) \cong \mathbb P(H^0(K_C+L-N)^*) 
:= \mathbb{P}$. When the natural map $\varphi:=\varphi_{|K_C+L-N|}: C\to\mathbb P$ is a morphism, set $X:=\varphi(C)\subset \mathbb P$.  
For any positive integer $h$ denote by $\Sec_h(X)$ the $h^{st}$-{\em secant variety} of $X$, 
defined as the closure of the union of all linear subspaces $\langle \varphi(D)\rangle\subset\mathbb P$,
for general effective divisors $D$ of degree $h$ on $C$. One has $\dim\;\Sec_h(X)\ =\min\,\{\dim\:\mathbb P,\;2h-1\}$. Recall: 
\begin{theorem} (\cite[Proposition 1.1]{LN})\label{LN}
Let $2\delta-d\ge 2$;  then $\varphi$ is a morphism and, for any integer
$$s \equiv 2\delta-d\  \text{ (mod\ 2) } \; \; \; \mbox{such that} \; \; \; 4+ d-2\delta\le s\le 2\delta-d,$$one has
$$ s (\mathcal F_u)\ge s \Leftrightarrow [(u)]\notin \Sec_{\frac{1}{2}(2\delta-d+s-2)}(X). $$
\end{theorem}

\section{Proof of Theorem \ref{thm:main2}}\label{S:proof} This section will be devoted to the proof of Theorem \ref{thm:main2}, which will be done in several steps (cf.\;\S's\;\ref{ss:extension},\;\ref{ss:regular},\;\ref{ss:superabundant} below).

\begin{remark}\label{rem:important} {\normalfont Notice first that, when $C$ is a general $\nu$-gonal curve, then $B^{k_2}_d \cap U^s_C(2,d)$:

\vskip 2pt

\noindent 
$(a)$ is empty, for $d = 4g-4,\; 4g-5, \;4g-6$ (cf. Remark \ref{rem:TeixRes}); 

\vskip 2pt

\noindent 
$(b)$ consists only of the component  $\mathcal B$ (of expected dimension $\rho_d^{k_2} = 8g-2d-11$) in Theorem \ref{TeixidorRes}, for any $4g-4-2\nu \leq d \leq 4g-7$. Indeed, conditions in 
Theorem \ref{TeixidorRes} guaranteeing reducibility are:  
$$2n < 4g-4-d, \;\; W^1_n \neq \emptyset \;\; {\rm and} \; \; \dim\;W^1_n  \geq 2g + 2n - d - 5.$$One must have $2\nu \leq 2 n$, since $C$ has no $g^1_n$ for $n < \nu$ (cf. \cite{AC}). Therefore 
$2 \nu \leq 2 n < 4g-4-d$ forces $d \leq 4g-5-2\nu$, which explains why $B_d^{k_2} \cap U^s_C(d)$ must be irreducible for $d \geq 4g-4-2\nu$. 
}
\end{remark}

\noindent
The previous remark motivates why we focus on  $d$ as in  \eqref{eq:ourbounds} in our Theorem \ref{thm:main2}.  

Before proving it, observe first \color{black} that its direct consequence is the following.\color{black}



\begin{corollary}\label{fixeddeterminant} With assumptions as in Theorem \ref{thm:main2}, let $M \in {\rm Pic}^d(C)$ be general, then the Brill-Noether locus $B_M^{k_2}(C)$, parametrizing 
semistable rank 2 vector bundles of given determinant $M$, with at least $k_2 = d-2g+4$ independent global sections, is not empty, even if its expected dimension 
$$\rho_M^{k_2} := 3g-3 - 2k_2= 3g-3 - 2(d-2g+4)$$is negative for $d > \frac{7g-11}{2}$
\end{corollary}
\begin{proof} Take $\mathcal F \in B_{\rm sup}$ general, as in Theorem \ref{thm:main2} (ii). From \eqref{exactB1}, one has $$\det(\mathcal F) \cong K_C-A+N,$$which is general in 
${\rm Pic}^d(C)$, since $N \in {\rm Pic}^{d-2g+2+\nu}(C)$ is general by assumption in Theorem \ref{thm:main2} (ii). Therefore, the determinantal map  
$$B_{\rm sup} \stackrel{det}{\longrightarrow} {\rm Pic}^d(C)$$is dominant.  Thus $B_M^{k_2}(C)$ contains $B_{\rm sup} \cap U_C(2,M)$, where $U_C(2,M)$ the moduli space of semistable vector bundles of determinant $M$ 
(i.e. the fiber in $U_C(2,d)$ over $M$ via the map $det$). 

Notice that $\dim\;B_{\rm sup} \cap U_C(2,M) = 5g-6-2\nu-d > 0$, since $M= K_C - A + N \cong M':= K_C - A + N'$ if and only if $N \cong N'$ (see Remark \ref{rem:fixeddeterminant} for the classification of $B_{\rm sup} 
\cap U_C(2,M)$). 
\end{proof}

\begin{remark} \label{rem:fixeddeterminant} {\normalfont (i) From the construction of $B_{\rm sup}$ conducted in \S\;\ref{ss:superabundant}, it will be clear that \linebreak $B_{\rm sup} \cap U_C(2,M)$ is birational to $\mathbb{P}(\ext^1(K_C-A, N))$, where $M \cong K_C-A+N \in {\rm Pic}^d(C)$ general.

\noindent
(ii) Differently to $B_{\rm sup}$, the component $B_{\rm reg}$ cannot dominate ${\rm Pic}^d(C)$ for $d >\frac{7g-11}{2}$; indeed, in such a case $$\dim\;B_{\rm reg} = 8g-2d-11 < g = \dim \; {\rm Pic}^d(C).$$
Finally, by Theorem \ref{TeixidorRes}, if  $C$ is with general moduli and $M \in {\rm Pic}^d(C)$ is general, with $d >  \frac{7g-11}{2}$ then $B_M^{k_2}(C) = \emptyset$. In view of the {\em residual correspondence} as in Remark \ref{rem:TeixRes}(ii), this also implies that for 
$d < \frac{g+3}{2}$ then $B_M^2(C) = \emptyset$ for $M \in  {\rm Pic}^d(C)$ and $C$ general. This extends to even degrees $d$ (and via completely different methods) what found in \cite[Example\;6.2]{LNS}.
}
\end{remark}

\subsection{Components via extensions}\label{ss:extension} To prove Theorem \ref{thm:main2} we make use of Theorem \ref{CF} from which we know that, for any possible irreducible component of $B_d^{k_2}$, 
its general bundle $\mathcal F$ arises as an extension \eqref{degree}, with $h^1(L)>0$.  The following preliminary result in particular restricts the possibilities for $h^1(L)$. 

\begin{lemma}\label{speciality3} There is no irreducible component of $B_d^{k_2}$ whose general member $\mathcal F$ is of speciality $i := h^1(\mathcal F) \geq 3$. 
\end{lemma}
\begin{proof} If $\mathcal F \in B_d^{k_2}$ is such that $h^1(\mathcal F) = i \ge 3$, 
then by the Riemann-Roch theorem $h^0(\mathcal F) = d - 2g + 2 + i = k_2 +(i-2) = k_i > k_2$. Thus $\mathcal F \in {\rm Sing} (B_d^{k_2})$ (cf. \cite[p.\;189]{ACGH}). Therefore the statement follows from 
\cite[Lemme 2.6]{L}, from which one deduces that no component 
of  $B_d^{k_2}$ can be entirely contained or coincide with a component of $B_d^{k_i}$, for any $i \geq 3$ (the proof is identical to that in \cite[pp.101-102]{L} for $B^0_d$, $1 \leq d \leq 2g-2$, which uses elementary transformations of vector bundles). 
\end{proof}

>From the previous lemma, a general element $\mathcal F$ of any possible component of $B_d^{k_2} $ is therefore presented via an extension \eqref{degree} for which either 
$$ h^1(L) =1 \ \mbox{ or } \  h^1(L) =2.$$The proof of Theorem \ref{thm:main2} splits into the following two subsections\;\S\;\ref{ss:regular} and \ref{ss:superabundant}, respectively dealing with the case 
$h^1(L) =1$ and $h^1(L) =2$. We will show in particular that the case $h^1(L) =1$ (resp., $h^1(L) =2$) produces only the component $B_{\rm reg}$ (resp., $B_{\rm sup}$) as in Theorem \ref{thm:main2}.


 \subsection{The regular component $B_{\rm reg}$} \label{ss:regular} In this subsection we prove that there exists only one component of $B_d^{k_2} \color{black} \cap U_C^s(2,d) \color{black}$, whose general bundle $\mathcal F$ is presented via an extension 
\eqref{degree} with $h^1(L) =1$ and this component is exactly $B_{\rm reg}$ as in Theorem \ref{thm:main2} $(i)$.

To do this recall first that, for any exact sequence $(u)$ as in \eqref{degree},  if one sets$$\partial_u : H^0(L) \to H^1(N)$$ 
the corresponding coboundary map then,  for any integer $t>0$, the locus 
\begin{equation}\label{W1}
\mathcal W_t:=\{u\in\ext^1(L,N)\ |\ {\rm corank} (\partial_u)\ge t\}\subseteq \ext^1(L,N), 
\end{equation}has a natural structure of determinantal scheme (cf. \cite[\S\,5.2]{CF}).  Observe further that, by \eqref{eq:ourbounds},  the Brill-Noether locus $W^0_{4g-5-d}$ on\;$C$ is not empty, irreducible and $h^0(D) =1$ for  general 
$D \in W^0_{4g-5-d}$, \color{black} where $\deg\;D = 4g-5-d \leq g$.\color{black} 


\begin{lemma}\label{lem:i=1.2} Let $D \in W^0_{4g-5-d}$ and  $p\in C$ be general and let 
$\mathcal W_1 \subseteq \ext^1(K_C-p,K_C-D)$ be as in \eqref{W1}. Then, $\mathcal W_1$ is a sub-vector space of $\ext^1(K_C-p,K_C-D)$ of dimension $4g-6-d$. 
Moreover, for $u\in \mathcal W_1$ general, the corresponding rank $2$ vector bundle $\mathcal F_u$ is stable, with:
\begin{enumerate}
\item[(a)] $h^1(\mathcal F_u)=2$;
\item[(b)] $s(\mathcal F) \ge 1$ (resp., $2$)  if $d$ is odd (resp., even).   
\end{enumerate} 
\end{lemma}

\begin{proof} This is a simplified and more general version of \cite[Proof of Lemma\;3.7]{CFK}. First we prove the assertions on $\mathcal W_1$; from the assumptions we have: 
\begin{equation}\label{degree0}
\begin{CD}
&(u)& : 0\to &K_C-D&\to &\ \ \mathcal F\ \ & \to  \ &\ \ K_C-p\ \ &\to 0\\
&\deg&       &d-2g+3&& d&& 2g-3&\\
&h^0&       &d-3g+5&&  && g-1&\\
&h^1&       &1&&  && 1&.
\end{CD}
\end{equation} Notice that $\mathcal W_1 \neq \emptyset$, as 
$(K_C-p) \oplus (K_C-D) \in \mathcal W_1$, and that $u \in \mathcal W_1$ if and only if 
$\partial_u = 0$, since $h^1(K_C-D) = 1$. Recalling  that $\partial_u$ is induced by the cup-product 
$$\cup: H^0(K_C-p) \otimes H^1(p-D) \to H^1(K_C-D),$$we have the natural induced morphism 
\[
\begin{array}{ccc}
H^1(p-D) & \stackrel{\Phi}{\longrightarrow} & {\rm Hom} \left(H^0(K_C-p),\; H^1(K_C-D)\right) \\
u & \longrightarrow & \partial_u
\end{array}
\]which shows that 
{\small
$${\mathcal W}_1 = \ker \left(H^1(p-D) \stackrel{\Phi}{\to} H^0(K_C-p)^{\vee} \otimes H^1(K_C-D)\right) 
\cong \left(\coker \left(H^0(K_C+D-p) \stackrel{\Phi^{\vee}}{\leftarrow} H^0(K_C-p) \otimes H^0(D)\right)\right)^{\vee}.$$
}Therefore $\mathcal W_1 = \left(\im \; \Phi^{\vee}\right)^{\perp}$ is a sub-vector space of $\ext^1(K_C-p,K_C-D)$. \color{black}  Since $h^0(D) =1$,   the morphism $\Phi^{\vee}$ is injective hence 
$\mathcal W_1$ is of codimension $(g-1)$. \color{black}  From \eqref{degree0} and definition of $\mathcal W_1$, any $u \in \mathcal W_1$ gives $h^1(\mathcal F_u) =2$, which 
in particular proves (a).

To show that, for $u \in \mathcal W_1$ general,  the bundle $\mathcal F_u$ satisfies also (b) we follow a similar strategy as in the proof of \cite[Lemma\;3.7]{CFK}. Precisely, we consider the 
linear subspace $$\widehat{\mathcal W}_1 := \mathbb{P} (\mathcal W_1) \subset\mathbb P:= \mathbb{P}\left(\ext^1(K_C-p, K_C-D)\right)$$which has dimension $4g-7-d$.

Consider the natural morphism $ C\stackrel{\varphi}{\longrightarrow} X\subset\mathbb P$, given by the complete linear system 
$|K_C+D-p|$, and take $X = \varphi\left(C\right)$. Posing $\delta := 2g-3$ and considering \eqref{eq:ourbounds}, one has $2\delta - d \geq 2 \nu -1 \geq 5$; thus 
one can apply  Theorem \ref{LN}. Taking therefore $s$ any integer such that  $ s\equiv 2\delta-d\  {\rm{ (mod \;2) \ \ and  }} $ 
$ 0 \le s\le 2\delta - d$ one has $$\dim \Sec_{\frac{1}{2}(2\delta-d+s-2)}(X) =2\delta - d + s - 3 = 4g - 9 - d + s \le 4g-7-d = \dim \widehat{\mathcal W}_1$$if and only if $s\le 2$, 
where the equality holds if and only if $s=2$. 

Thus, for $d$ odd,  $s(\mathcal F_u) \geq 1$ for $u \in \mathcal W_1$ general by Theorem \ref{LN}. For $d$ even (case in which $s$ has to be taken necessarily equal to $2$ and $\dim \widehat{\mathcal W}_1 = \dim \Sec_{\frac{1}{2}(4g-6-d)}(X) = 4g-7-d$) 
one has $\widehat{\mathcal W}_1 \neq  \Sec_{\frac{1}{2}(4g-6-d)}(X)$ since $\widehat{\mathcal W}_1$ is a linear space whereas $\Sec_{\frac{1}{2}(4g-6-d)}(X)$ is non-degenerate as $X \subset \mathbb P$ is not;  thus, by Theorem \ref{LN}, 
in this case for general $u \in \widehat{\mathcal W}_1$ one has $s(\mathcal F_u ) \geq 2$.  

In any case  $\mathcal F_u$ general is stable and satisfies (b). 
\end{proof}

To construct the locus $B_{\rm reg}$ and to show that it is actually a component of $B_d^{k_2}$ as in Theorem \ref{thm:main2} (i), notice first that as in \cite[Sect.\;3.2]{CFK} 
one has a natural projective bundle $\mathbb P(\mathcal E_{d})\to S$, where \linebreak $S \subseteq W^0_{4g-5-d}\times C$ is a suitable open dense subset; namely, 
$\mathbb P(\mathcal E_{d})$ is the family of \linebreak $\mathbb P(\ext^1(K_C-p,K_C-D))$'s  as  $(D, p) \in S$ varies. Since, for any such $(D, p) \in S$, 
$\widehat{\mathcal W}_1$ is a linear space of (constant) dimension $4g-7-d$, one has an irreducible projective variety 
$$\widehat{\mathcal W}_1^{Tot}:= \left\{ (D,p,u) \in \mathbb P(\mathcal E_{d}) \; | \; H^0(K_C-p) \stackrel{\partial_u = 0}{\longrightarrow} H^1(K_C-D)\right\},$$which is ruled over $S$, of 
dimension 
$$\dim \widehat{\mathcal W}_1^{Tot} = \dim S + 4g - 7 - d = 4g - d - 4 + 4g - 7 - d = 8 g - 2d - 11 = \rho_d^{k_2}.$$
From Lemma \ref{lem:i=1.2}, one has the natural (rational) map
 \[
 \begin{array}{ccc}
 \widehat{\mathcal W}_1^{Tot}& \stackrel{\pi}{\dashrightarrow} & U^s_C(2,d) \\
 (D,p, u) & \longrightarrow &\mathcal F_u
 \end{array}
 \] and  $\im(\pi) \subseteq B^{k_2}_d \cap U_C^s(2,d)$.

\begin{proposition}\label{thm:i=1.3} The closure $B_{\rm reg}$ of $\im(\pi)$ in $U_C(2,d)$ is a generically  smooth component of  $B^{k_2}_d \color{black} \cap U^s_C(2,d)$  \color{black} with (expected) dimension $\rho_d^{k_2} = 8g-11-2d$, i.e. $B_{\rm reg}$ is {\em regular}. 
Moreover, $B_{\rm reg}$ is {\em uniruled}, being finitely dominated by $\widehat{\mathcal W}_1^{Tot}$. The general point of $B_{\rm reg}$ arises as in Lemma \ref{lem:i=1.2}. 
\end{proposition}

\begin{proof} The proof of the first sentence is identical to that in \cite[Proposition\;3.9]{CFK}. The fact that $B_{\rm reg}$ is uniruled follows from the ruledness of $\widehat{\mathcal W}_1^{Tot}$ and the 
\color{black}generic \color{black} finiteness of the map $\pi$  (as it is proved in \cite[Lemma\;6.2]{CF}, cf. also \cite[Proposition\;3.9]{CFK}).
\end{proof}

Next, we show the uniqueness of $B_{\rm reg} $ among possible components of $B_d^{k_2} \color{black} \cap U_C^s(2,d) \color{black}$, whose general bundle $\mathcal F$ is presented via an extension 
\eqref{degree} with $h^1(L) =1$. To do this,  we will make use of the following: 

\begin{theorem}(\cite[Theorem 5.8 and Corollary 5.9]{CF})\label{CF5.8} Let $C$ be a smooth, irreducible, projective curve of genus $g\ge 3$, $L \in \pic^{\delta}(C)$ and $N \in \pic^{d-\delta}(C)$. Set 
$$l := h^0(L), \;\; r:= h^1(N), \;\;  m:=\dim\;\ext^1(L,N).$$Assume that  
$$ r\ge1,\ l\ge\max\{1,r-1\},\ m \ge l+1.$$Then:
\begin{enumerate}
\item[(i)] $\mathcal W_1$ as in \eqref{W1} is irreducible of dimension $m- (l-r+1)$; 
\item[(ii)] if $l \geq r$, then $\mathcal W_1 \subsetneq \ext^1(L,N)$. Moreover for general $u \in  \ext^1(L,N)$,  $\partial_u$ is surjective whereas for general $w \in  \mathcal W_1$,  ${\rm corank} (\partial_w)=1$. 
\end{enumerate}
\end{theorem}

\begin{proposition}\label{lem:i=1.1} Let $\mathcal B$ be any component of $B^{k_2}_d $,  
with $\dim \mathcal B \geq \rho_d^{k_2}$.  Assume that a general element
$\mathcal F$ in $\mathcal B$ fits in \eqref{degree}, with $h^1(\mathcal F)=2$ and $h^1(L)=1$. Then, $\mathcal B$ coincides with the component $B_{\rm reg}$ as in Proposition\;\ref{thm:i=1.3}.
\end{proposition}
\begin{proof} The strategy of the proof is similar to that of \cite[Prop.\;3.13]{CFK}; the main difference is given by different bounds \eqref{eq:ourbounds}. 

Let $\delta := \deg \;L$; then, $\frac{3g-5}{2} \leq \delta\le 2g-2$, as it follows from the facts that $L$ is special and $\mathcal F$ is semistable  
with $d \geq 3g-5$ from \eqref{eq:ourbounds}.  Hence, \color{black} using \eqref{eq:ourbounds}, \color{black} one has 
\begin{equation}\label{degn}
g -3 \le \deg \; N=d-\delta\le \frac{d}{2} \le 2g - \frac{5}{2} - \nu.
\end{equation} By  \eqref{degree}, $h^1(\mathcal F) =2$ and $h^1(L)=1$, $N$ is therefore special with $r: = h^1(N) \geq 1$ and the corresponding coboundary map 
$\partial$ has to be of corank one.  

Set  $l:=h^0(L)$;  by $h^1(L) =1$ one has $l=\delta-g+2$. 

First we want to show that $l \geq r$; observe indeed that $3d \geq 9g - 15 \geq 8g-10$, where the first inequality follows from \eqref{eq:ourbounds} whereas the latter from $g \geq 2 \nu \geq 6$, always by 
\eqref{eq:ourbounds}. Therefore$$3d \geq 8g -10 \; \Rightarrow \; d \geq 8g - 2d - 10 \; \Rightarrow \frac{d}{2} \geq 4g - d - 5.$$By semistability of $\mathcal F$, the last inequality in particular implies 
$\delta \geq \frac{d}{2} \geq 4g - d - 5$, which is equivalent to 
$l= \delta - g + 2 \geq \frac{2g-1-d+\delta}{2} \geq r$, the last inequality following from  $r-1 \leq \frac{\deg(K_C-N)-1}{2}$ by Clifford's theorem, as $C$ is non--hyperelliptic. 

Now set $m:=\dim \; {\rm Ext}^1(L,N)$; we want to prove that  $m\ge l+1$. From \eqref{eq:m}, one has \linebreak $m\geq g+2\delta-d -1$ so 
it suffices to show that $g+2\delta-d -1 \geq \delta - g + 3$. This is equivalent to $d-\delta \leq 2g-4$, which certainly holds since 
$\deg\;N = d-\delta \leq \frac{d}{2} \leq 2g - \frac{5}{2} - \nu$ as it follows from semistability and from \eqref{eq:ourbounds}. 

To sum--up, since $l \geq r$ and $m \geq l +1$, we are in position to apply Theorem \ref{CF5.8}, from which we get that 
$$\emptyset \neq \mathcal W_1=\{u\in\ext^1(L,N)\ |\ {\rm corank} (\partial_u)\ge1\} \subsetneq {\rm Ext}^1(L,N)$$ is irreducible and $ \dim {\mathcal W}_1 = m-l+r-1 =  m - \delta +g +r - 3$. 
Using the same strategy as above (cf. also the proof of \cite[Prop.\;3.13]{CFK}), for a suitable open dense subset $S \subseteq W^{r-1}_{2g-2 +\delta -d}\times C^{(2g-2-\delta)}$, 
one can construct a projective bundle $\mathbb P(\mathcal E_d)\to S$ and an irreducible subvariety 
$\widehat{\mathcal W}^{Tot}_1 \subsetneq \mathbb P(\mathcal E_d)$, fitting in the diagram:
\[\begin{array}{ccc}
\widehat{\mathcal W}^{Tot}_1 & \stackrel{\pi}{\dasharrow} & \mathcal B \subset B_{d}^{k_2}\\
\downarrow & & \\
S & & 
\end{array}
\]whose general fiber over $S$ is $\widehat{\mathcal W}_1:= \mathbb P(\mathcal W_1)$, which is the projectivization of the affine irreducible variety $\mathcal W_1 \subsetneq {\rm Ext}^1(L,N)$, 
and such that the component $\mathcal B$ has to be the image of $ \widehat{\mathcal W}^{Tot}_1$  via a dominant rational map $\pi$ as above (cf.\;\cite[Sect.\;6]{CF} for details).  
From the given parametric construction of $\mathcal B$, one must have $$ \dim  \mathcal B \leq \dim W^{r-1}_{2g-2-d+\delta}+2g-2-\delta+\dim \widehat{\mathcal W}_1.$$

Observe that, from \eqref{degn}, one has $\deg \; K_C-N \le g+1$. To conclude the proof for 
$\deg \; K_C-N \le g-1$ one can refer to \cite[proof of Prop.\;3.13]{CFK}. Assume therefore $\deg \; K_C-N= g+a$ where $a\in\{0,1\}$; 
thus $\deg\; N = d-\delta = g-2-a$. 

If $r\ge a+2$, then we have $h^0(N) = r-a-1\ge 1$, hence $N\in W^{r-a-2}_{g-2-a} \subsetneq \text{Pic}^{g-2-a}(C)$; otherwise, if 
$r= a+1$, by $\deg\; N = g-2-a$ one deduces that $N\in \text{Pic}^{g-2-a}(C)$ is general.  
Hence we get
\begin{eqnarray*}
\dim\mathcal B  \leq  \begin{cases}
\ \dim \text{Pic}^{g-2-a}(C)+(2g-2-\delta)+ m - \delta +g +r - 4\ &\text{ if } r= a+1 \\
\ \dim W^{r-a-2}_{g-2-a}+(2g-2-\delta)+ m - \delta +g +r - 4\ &\text{ if } r\ge a+2.
\end{cases}
\end{eqnarray*}

Consider the second case $r\ge a+2$; since $r\ge 2$ then $N$ cannot be isomorphic to $L$ which, from \eqref{eq:m}, implies $m=2\delta -d+g-1$.
Hence from above we have
\begin{eqnarray*}
\dim\mathcal B & \leq & \ \dim W^{r-a-2}_{g-2-a}+(2g-2-\delta)+ m - \delta +g +r - 4\ \\
& \leq & (g-2-a)-2(r-a-2)-1 + (2g-2-\delta)+ (2\delta -d+g-1) - \delta +g +r - 4 \\
& = & 5g-6+a-r-d = 6g-2d+\delta-8-r,
\end{eqnarray*}where the second inequality follows from Martens' theorem \cite[Theorem \color{black}\;(5.1)]{ACGH} applied to $N$ whereas 
the last equality comes from $g=d-\delta+2+a$. This gives $ \rho^{k_2}_d =8g - 2d - 11  \leq \dim  \mathcal B \le 6g-2d+\delta-8-r$,  
which cannot occur since $\delta \leq 2g-2$.

Assume now $r= a+1$. If  $L \cong  N$, then $m=g$ by \eqref{eq:m}, so  
\begin{eqnarray*}
\dim \; \mathcal B  \leq g+(2g-2-\delta)+ m - \delta +g +r - 4 = 5g - 2 \delta - 5 + a = 6g-d-\delta-7 ,
\end{eqnarray*}where the last equality follows from $g= d - \delta + 2 + a$. Therefore, from 
$\rho_d^{k_2} \leq \dim \;\mathcal B  \leq  6g-d-\delta-7$ one would have $\deg \; N = d - \delta \geq 2g-4$ which is a contradiction from \eqref{degn}.

If otherwise $L\ncong  N$, then 
\begin{eqnarray*}
\dim\mathcal B  \leq g+(2g-2-\delta)+ m - \delta +g +r - 4=6g-2d+\delta-8,
\end{eqnarray*}where the last equality follows from \eqref{eq:m} and $g= d - \delta + 2 + a$.    As above,  from 
$\rho_d^{k_2} \leq \dim \;\mathcal B  \leq  6g-2d +\delta-8$, one gets  $\delta \geq 2g-3$ 
which implies that either $L\cong K_C$ or $L\cong K_C(-p)$, for some $p\in C$. Then one concludes as in the last part of the proof of \cite[Prop.\;3.13]{CFK}
\end{proof}

\begin{remark}\label{rem:min} {\normalfont The proof of Proposition \ref{lem:i=1.1} shows that $K_C-p$ is minimal among special quotient line bundles for 
$\mathcal F$ general in $B_{\rm reg}$, completely proving Theorem \ref{thm:main2} (i). Note moreover that \eqref{exactB0} implies that $\mathcal F$ general in  $B_{\rm reg}$ admits also a {\em presentation} via a canonical quotient, i.e. it fits in  $0 \to K_C-D-p \to \mathcal F \to K_C\to 0 $,  whose residual presentation coincides with that in the proof of \cite[Theorem]{Teixidor1}. 
In other words, the component  $B_{\rm reg}$ coincides with the component $\mathcal B$ in  \cite[Theorem]{Teixidor1}; this is the only component when $C$ is with general moduli. 
}
\end{remark}

\subsection{The superabundant component $B_{\rm sup}$}\label{ss:superabundant}

To finish the proof of Theorem \ref{thm:main2}, it remains to study possible components $\mathcal B$ for which $\mathcal F\in \mathcal B $ general is such that $h^1(\mathcal F) =h^1(L)=2$, with $\mathcal F$ fitting in a suitable exact sequence as 
in \eqref{degree}. To do this, we first need the following:

\begin{lemma}\label{lem:i=2.2} Let $\mathcal F$ be a rank 2 vector bundle arising as a general extension in ${\rm Ext}^1(K_C-A, N)$, where $N$ is any line bundle in ${\rm Pic}^{d-2g+2 + \nu}(C)$, with $d$ and $\nu$ as in \eqref{eq:ourbounds}. 
Then: 
\begin{enumerate}
\item[(a)]  $\mathcal F$ is stable with $s(\mathcal F)= 4g-4-2\nu-d$, i.e. $K_C-A$ is a minimal quotient of  $\mathcal F$; 
\item[(b)] If moreover $N$ is non special, then $h^1(\mathcal F) = h^1(K_C-A) = 2$. 
\end{enumerate} 
\end{lemma}

\begin{proof} (b) is a trivial consequence of the exact sequence $0 \to N \to \mathcal F \to K_C-A \to 0$ and the assumption on $N$; in particular, for any $u \in {\rm Ext}^1(K_C-A, N)$, one has 
$h^1(\mathcal F_u) = 2$. 

To prove (a), in order to ease notation, we set $L:= K_C-A$ and $\delta:=\deg\;L = \deg \; K_C -A = 2g-2-\nu$.  

\noindent
$\bullet$ For $3g-5 \leq d \leq 4g-6 - 2 \nu$, one can reason as in the proof of \cite[Theorem 3.1]{CFK}. Indeed, the upper bound on $d$ implies $2\delta-d = 2 (2g-2-\nu) - d \ge 2$, so  
one can apply  Theorem \ref{LN} with $s = 2\delta - d$ and $ C\stackrel{|K_C+L-N|}{\longrightarrow} X\subset\mathbb P:=\mathbb P(\ext^1(L,N))$. With these choices, one has 
$$\dim\left(\Sec_{\frac{1}{2}(2(2\delta-d)-2)}(X)\right)=2(2\delta-d)-3<2 \delta-d+g-2= \dim\;\mathbb P,$$where the last equality follows from \eqref{eq:m}  
and the fact that $L = K_C-A \ncong N$, as $\deg \; L-N = 2 \delta - 2 \geq 2$, whereas the strict inequality in the middle follows from \eqref{eq:ourbounds}, as 
$2\delta-d=4g-4-2 \nu -d \leq g +1 - 2 \nu \leq g-5$. Thus, $\mathcal F= \mathcal F_u$ arising from 
$u \in {\rm Ext}^1(K_C-A, N)$ general is of degree $d$ and stable, 
since $s(\mathcal F_u) = 2 \delta - d = 4g-4 - 2\nu -d \geq 2$; the equality $s(\mathcal F_u) = 2 \delta - d = 4g-4 - 2\nu -d$ follows from Theorem \ref{LN} and 
the fact that $u  \in {\rm Ext}^1(K_C-A, N)$.

\vskip 2pt

\noindent
$\bullet$ For $d = 4g-5-2\nu$, Theorem \ref{LN} does not apply, as in this case one has $2 \delta - d =1$. On the other hand, since $d$ is odd, to prove stability of $\mathcal F=\mathcal F_u$ general as above is equivalent to show that $\mathcal F_u$ 
is not unstable. Assume, by contradiction there exists a sub-line bundle $M \hookrightarrow \mathcal F_u$  such that $\deg\; M \geq 2 g - 2 - \nu > \frac{d}{2}$. We would get therefore the
following commutative diagram:
\[
\xymatrix@C-2mm@R-2mm{ & & 0 \ar[d] \\
& & M \ar[d] \ar[dr]^{\varphi} \\
0 \ar[r] & N \ar[r] & \mathcal F_u \ar[r] & K_C-A \ar[r] & 0.
}
\]Since $\deg \; N = 2g - 3 - \nu$, $\varphi$ is not the zero-map. On the other hand, $\varphi$ can be neither strictly injective (for degree 
reasons) nor an isomorphism (otherwise $\mathcal F_u \cong N \oplus (K_C-A)$, contradicting the generality of $u  \in {\rm Ext}^1(K_C-A, N)$).
\end{proof}

Now we can prove that $B_{\rm sup}$ as in Theorem \ref{thm:main2} (ii) is a component of  $B^{k_2}_d$. The definition of the locus $B_{\rm sup} \subset B_d^{k_2} \cap U_C^s(2,d)$ follows from Lemma \ref{lem:i=2.2} and the 
construction in \cite[\S\;3.1]{CFK}, which still works under condition \eqref{eq:ourbounds}; precisely, using the diagram after \cite[Lemma\;3.3]{CFK}, one can consider a vector bundle 
$\mathcal E_{d,\nu}$ on a suitable open, dense subset $S \subseteq \pic^{d-2g+2+\nu}(C)$,  whose rank is  
$ \dim{\ext^1(K_C-A,N)}= 5g-5- 2 \nu -d $ as in \eqref{eq:m}, since $K_C-A \ncong N$ (cf.\;\cite[pp.\;166-167]{ACGH}). 
Taking the associated projective bundle $\mathbb P(\mathcal E_{d,\nu})\to S$ (consisting of the family of $\mathbb P\left(\ext^1(K_C-A,N)\right)$'s as  $N$ varies in $S$) 
one has$$ \dim \mathbb P(\mathcal E_{d,\nu}) 
=g+ (5g-5- 2 \nu -d) -1=6g-6-2 \nu  -d. $$From Lemma \ref{lem:i=2.2}, one has a natural (rational) map
 \begin{eqnarray*}
 &\mathbb P(\mathcal E_{d,\nu})\stackrel{\pi_{d,\nu}}{\dashrightarrow} &U_C(2,d) \\
 &(N, u)\to &\mathcal F_u;
 \end{eqnarray*} which  gives $ \im (\pi_{d,\nu})\subseteq B^{k_2}_d \cap  U^s_C(2,d)$. Once we show that $\pi_{d,\nu}$ is birational 
onto its image, we will get that the closure $B_{\rm sup}$ of $\im (\pi_{d,\nu})$ in $U_C(2,d)$ is {\em ruled}, being 
birational to $ \mathbb P(\mathcal E_{d,\nu})$ which is ruled over $\pic^{d-2g+2+\nu}(C)$, and such that $\dim \; B_{\rm sup} = 6g - 6 - 2\nu - d$.

\begin{claim}\label{cl:birat} $\pi_{d,\nu}$ is birational onto its image. 
\end{claim}
\begin{proof}[Proof of Claim \ref{cl:birat}] Let $\Gamma \subset F := \mathbb{P}(\mathcal F_u)$ be the section of the ruled surface $F$ corresponding to the quotient $\mathcal F_u \to\!\!\!\!\to K_C-A$. 
$\Gamma$ is the only section of degree $2g-2-\nu$ and speciality $2$ of $F$, since $K_C-A$ is the only line bundle with these properties on $C$. Moreover $\Gamma$ is also linearly isolated. 
This guarantees that  $\pi_{d,\nu}$ is birational onto its image (for more details see the proof of \cite[Lemma\;6.2]{CF}). 
\end{proof}

Now we can show the following:

\begin{theorem}\label{lem:i=2.1} Under assumptions \eqref{eq:ourbounds}, $B_{\rm sup}$ is an irreducible  component of $B^{k_2}_d \color{black} \cap U^s_C(2,d)$ \color{black} which is {\em superabundant}. Moreover, it is:  
\begin{enumerate} 
\item[(i)] generically smooth, if $d\geq 3g-3$,
\item[$(ii)$] non-reduced, if $d=3g-4, \ 3g-5$. 
\end{enumerate}
\end{theorem}
\begin{proof} The result will follow once we prove that, for general $\mathcal F \in B_{\rm sup}$,
\begin{equation} \label{tan_sup} \dim T_{\mathcal F}(B^{k_2}_d)=\begin{cases}\dim B_{\rm sup}  &\mbox{  if  } d\geq 3g-3\\ 
\dim B_{\rm sup}  +3 g-3 -d &\mbox{  if  } d= 3g-4, \ 3g-5\end{cases} 
\end{equation}and moreover, for $d = 3g-4,\;3g-5$, $B_{\rm sup}$ is actually a component of $B^{k_2}_d$. 

Concerning tangent space computations, one can consider the Petri map of a general $\mathcal F\in B_{\rm sup}$:
$$ \mu_\mathcal F: H^0(\mathcal F)\otimes H^0(\omega_C\otimes \mathcal F^*)\to H^0(\omega_C\otimes \mathcal F\otimes \mathcal F^*). $$Since, by construction of $B_{\rm sup}$ as a birational image of $\mathbb P(\mathcal E_{d,\nu})$, 
$\mathcal F$ general fits in an exact sequence as  \eqref{exactB1}, with $N \in {\rm Pic}^{d-2g+2 + \nu}(C)$ general; by \eqref{eq:ourbounds} one has therefore $h^1(N)=0$. Thus, we have
$$H^0(\mathcal F)\simeq H^0(N)\oplus H^0(K_C-A) \;\;\; {\rm and} \;\;\; 
H^0(\omega_C\otimes \mathcal F^*)\simeq H^0(A).$$In particular, $\mu_\mathcal F$ reads as 
\begin{equation*}
\begin{CD}
\left(H^0(N)\oplus H^0(K_C-A)\right)&\;\otimes\;& \;H^0(A)&\;\; \stackrel
{\mu_\mathcal F }{\longrightarrow}\;\; &H^0(\omega_C\otimes \mathcal F\otimes \mathcal F^*).\\
\end{CD}
 \end{equation*}  Consider the following natural multiplication maps:
 \begin{eqnarray}
 \mu_{A,N}:& H^0(A)\otimes H^0(N)\to H^0(N+A)\label{muA}\\
 \mu_{0,A}: & H^0(A)\otimes H^0(K_C-A)\to H^0(K_C)\label{mu0}. 
 \end{eqnarray}
\begin{claim}\label{cl:ker}  $\ker(\mu_\mathcal F)\simeq \ker(\mu_{0,A})\oplus \ker(\mu_{A,N}) $. 
\end{claim}
\begin{proof}[Proof of Claim \ref{cl:ker}] The proof is a simplified and extended version of \cite[Proof of Claim 3.5]{CFK}.  Since $h^1(N) = h^1(N+A) = 0$, one has the following commutative diagram 

\begin{equation*}\label{eq1b}
{\small
\begin{array}{ccl}
0&&0\\[1ex]
\downarrow &&\downarrow\\[1ex]
\color{black} H^0(A) \otimes \color{black} H^0(N) & \stackrel{\mu_{A,N}}{\longrightarrow} & H^0(N+A)  \\[1ex]
\downarrow && \downarrow  \\[1ex]
 \color{black} H^0(A) \otimes \color{black} H^0(\mathcal F) & \stackrel{\mu_\mathcal F}{\longrightarrow} & H^0(\mathcal F \otimes A) \subset H^0 (\omega_C \otimes \mathcal F\otimes \mathcal F^*) \\[1ex]
\downarrow &&\downarrow \\[1ex]
H^0(A)\otimes H^0(K_C-A) & \stackrel{\mu_{0,A}}{\longrightarrow}   & H^0(K_C) \\[1ex]
\downarrow &&\downarrow\\[1ex]
0&&0
\end{array}
}
\end{equation*}where the column on the left comes from the $H^0$-cohomology sequence of \eqref{exactB1} tensored by $H^0(A)$, whereas the column on the right comes from  \eqref{exactB1} tensored by $A$ and then \color{black} taking \color{black} cohomology. 
The previous diagram proves the statement. 
\end{proof} 

\noindent 
By Claim \ref{cl:ker}, one has 
\begin{eqnarray*}
\dim T_{\mathcal F}(B^{k_2}_d)&=&4g-3-h^0(\mathcal F)h^1(\mathcal F)+\dim\;\ker\;\mu_\mathcal F \\
&=&4g-3-2(d-2g+4)+\dim\;\ker\;\mu_0(A)+\dim\;\ker\;\mu_{A,N}.
\end{eqnarray*} From \eqref{muA} and \eqref{mu0}, we have
\begin{equation}\label{eq:kers}
\ker(\mu_{0,A})\simeq H^0(K_C-2A)\cong H^1(2A)^*\;\; {\rm and} \;\;   \ker(\mu_{A,N})\simeq H^0(N-A),
\end{equation}as it follows from the base point free pencil trick. 

From \cite[Theorem\;(2.6)]{ACGH} and \cite[p. 869,\;Theorem\;(12.16)]{ACG}, one has 
\begin{equation}\label{eq:h12A}
h^1(2A)=g + 2 -2 \nu. 
\end{equation}As for $\ker(\mu_{A,N})$, the generality of $N$ 
implies that $N-A$ is general of its degree, which is $\deg \; N-A = \deg \; N - \nu =d-2g+2$. Therefore  it follows that 
$$\begin{cases}
h^1 (N-A) =0, \mbox{ equivalently, } h^0 (N-A) =d-3g+3, &\mbox{ for } d\geq 3g-3\\
 h^0 (N-A) = 0, &\mbox{ for } d= 3g-4, \ 3g-5. \end{cases}$$Thus we have
$$ \dim T_{\mathcal F}(B^{k_2}_d) =\begin{cases} 4g-3-2(d-2g+4)+(g + 2 -2 \nu) + (d-3g+3) &\mbox{ if } d\geq 3g-3\\
4g-3-2(d-2g+4)+(g + 2 -2 \nu) &\mbox{ if } d=3g-4, \ d=3g-5, 
\end{cases}$$
which gives \eqref{tan_sup} since $\dim B_{\rm sup} = 6g-6-2 \nu  -d.$

The fact that $B_{\rm sup}$ is actually a (non--reduced) component of $B_d^{k_2}$, for $d = 3g-4,\;3g-5$, will be a direct consequence of the previous computations and the next more general result.  
\end{proof}

In the next lemma we will prove that $B_{\rm sup}$ is the only component of 
$B_{d}^{k_2}$, for which the general bundle $\mathcal F$ is such that 
$h^1(\mathcal F) = h^1(L) =2$, for suitable $L$ as in \eqref{degree}. In particular, this will also imply that, for $d = 3g-4,\;3g-5$, $B_{\rm sup}$ cannot be strictly contained in a component of $B_{d}^{k_2}$, 
finishing the proof of  Theorem \ref{lem:i=2.1}.

\begin{lemma}\label{lemsup} Assume that $\mathcal B$ is any irreducible  component of $B^{k_2}_d$ such that a general $\mathcal F \in  \mathcal B$ fits in an exact sequence like \eqref{degree}, 
with $h^1(\mathcal F)=h^1(L)=2$.  Then $\mathcal B = B_{\rm sup}$. 
\end{lemma}

\begin{proof} Since $\mathcal F$ is  semistable, from \eqref{eq:neccond} and \eqref{eq:ourbounds} one has  
 $\deg \; L\;\ge \frac{3g-5}{2}$. Moreover, since $C$ is a general $\nu$-gonal curve and $h^1(L)=2$, from \cite[Theorem\;2.6]{AC} we have
$K_C - L \cong A +B_b$, where $B_b \in C^{(b)}$ is a base locus of degree $b \geq 0$. 

By assumption, $\mathcal F$ corresponds to a suitable $v \in {\rm Ext}^1(K_C-A-B_b, N_b)$, for some \linebreak $N_b \in {\rm Pic}^{d-2g+2+\nu+b}(C)$. \color{black} Moreover, always by assumption, $L = K_C-A-B_b$ is such that $h^1(L) = h^1(K_C-A-B_b)= h^1(\mathcal F) = 2$; therefore,  by  taking cohomology in 
$0 \to N_b \to \mathcal F \to K_C-A-B_b \to 0$, irrespectively \color{black} of the speciality of $N_b$, the corresponding 
coboundary map $H^0(K_C-A-B_b) \stackrel{\partial_v}{\longrightarrow} H^1(N_b)$ has to be surjective. From \color{black} semicontinuity on the (affine) space ${\rm Ext}^1(K_C-A-B_b, N_b)$ \color{black} and the fact that semistability is an open condition, it follows that for a general $u \in   {\rm Ext}^1(K_C-A-B_b, N_b)$ the coboundary map $\partial_u$ is surjective too and $\mathcal F_u$ is semistable of speciality $2$. Since $\mathcal B$ is by assumption a component of $B^{k_2}_d$ and since 
$u $ general specializes to $v \in {\rm Ext}^1(K_C-A-B_b, N_b)$, one has that 
$\mathcal F \in \mathcal B$ has to come from a general $u \in   {\rm Ext}^1(K_C-A-B_b, N_b)$, for some $B_b \in C^{(b)}$ and some $N_b \in {\rm Pic}^{d-2g+2+\nu+b}(C)$.  

On the other hand, a general  extension as 
$$(*)\;\;\; 0 \to N_b \to \mathcal F_u \to K_C-A-B_b \to 0$$ is a flat specialization of a general extension of the form 
$$(**)\;\;\; 0 \to N \to \mathcal F \to K_C-A \to 0,$$where $N \cong N_b -B_b$; indeed extensions $(**)$ are parametrized by 
$\ext^1(K_C-A, N) \cong H^1(N+A -K_C)$ whereas extensions $(*)$ are parametrized by  $\ext^1(K_C-A-B_b,N+B_b) \cong H^1(N +  2B_b + A - K_C)$ and the aforementioned  existence of such a flat specialization is granted by the surjectivity   
$$H^1(N+A -K_C) \twoheadrightarrow  H^1(N + 2 B_b + A - K_C),$$which follows from the exact sequence 
$0 \to \mathcal O_C \to \mathcal O_C(2 B_b) \to \mathcal O_{2B_b} \to 0$ tensored by $N+A -K_C$ (cf. \cite[pp.\;101-102]{L} for the use of {\em elementary transformations} of vector bundles to interpret the above surjectivity).

From Lemma \ref{lem:i=2.2} (a), semicontinuity and the construction of $B_{\rm sup}$, general extension $(**)$ gives rise to a point of 
$B_{\rm sup}$; by specialization of a general $(**)$ to a general $(*)$, one can conclude that $\mathcal B \subseteq B_{\rm sup}$, i.e. $\mathcal B =  B_{\rm sup}$. 
\end{proof}

\begin{remark}\label{rem:sup_reg} 
{\normalfont Notice that \eqref{eq:ourbounds} implies $\nu \leq \frac{g}{2}$; more precisely, when $\nu =\frac{g}{2}$, one can easily compute that the only admissible value for $d$ in \eqref{eq:ourbounds} is $d=3g-5$. 
In such a case, i.e. $(\nu, d) = (\frac{g}{2}, 3g-5)$, one has $\dim \; B_{\rm reg} = \dim \; B_{\rm sup} = \rho_{3g-5}^{g-1} = 2g-1$. On the other hand, 
for any $d$ and $\nu$ as in \eqref{eq:ourbounds}, \eqref{tan_sup} states that 
$\dim T_{\mathcal F_u}(B^{k_2}_d) > \rho ^{k_2} _d$ for general $\mathcal F_u \in B_{\rm sup}$ whereas, from Proposition \ref{thm:i=1.3},  $\dim T_{\mathcal F_v} (B^{k_2}_d) =\rho ^{k_2}_d $ for general 
$\mathcal F_v \in B_{\rm reg}$. Thus, for $(\nu, d) = (\frac{g}{2}, 3g-5)$, $B_{\rm reg}$ and $B_{\rm sup}$ are distinct irreducible components of $B^{g-1}_{3g-5}$, both of expected dimension, the first 
regular the second superabundant, being non-reduced. 
}
\end{remark}


\section{Very-ampleness of vector bundles as in Theorem \ref{thm:main2}}\label{S:va} In this section, we will find sufficient conditions guaranteeing that a general bundle in 
$B_{\rm reg}$ (respectively, in $B_{\rm sup}$) is very ample. 

Concerning the component $B_{\rm reg}$ we already observed that, as predicted also by Theorem \ref{TeixidorRes}, it makes sense also on $C$ with general moduli and it is actually 
the unique component of $B_{d}^{k_2} \cap U^s_C(2,d)$ on such a $C$. Construction and properties of $B_{\rm reg}$ in this case are similar to those conducted in \S\;\ref{ss:regular}. 
We will find therefore very--ampleness conditions for a general $\mathcal F \in B_{\rm reg}$ also for $C$ with general moduli, since we will use this condition in \S\;\ref{ss:Hreg}.

As for $B_{\rm sup}$ on a general $\nu$--gonal curve $C$ as above, in order to find sufficient very--ampleness conditions for $\mathcal F \in B_{\rm sup}$ general, 
we will make use of the following: 

\begin{lemma}\label{S2} (cf. \cite[Corollary\;1]{KK}) On a general $\nu$-gonal curve $C$ of genus $g \ge 2\nu-2$, with 
$\nu \ge 3$, there does not exist a $g^r_{\nu - 2 + 2r}$ with $\nu - 2 + 2r \le g-1$, $r \ge 2$.  
\end{lemma}

\begin{theorem}\label{thm:veryampleness} Take notation as in Theorem \ref{thm:main2}. 

\noindent
(i) If $C$ is a general $\nu$--gonal curve, with $d$ and $\nu$ as in \eqref{eq:ourbounds}, a general $\mathcal F \in B_{\rm reg}$ is very ample for $\nu \ge 4$ and $d \ge 3g-1$. 
If $C$ is a curve with general moduli, a general $\mathcal F \in B_{\rm reg}$ is very ample for $d \ge 3g-1$.

\vskip 2pt

\noindent
(ii) If $C$ is a general $\nu$--gonal curve, with $d$ and $\nu$ as in \eqref{eq:ourbounds}, a general $\mathcal F \in B_{\rm sup}$ is very ample for $d+\nu \ge 3g+1$.

\end{theorem}

\begin{proof} (i) When $C$ is a general $\nu$-gonal curve, for $d$ and $\nu$ as in \eqref{eq:ourbounds}, the strategy of \cite[Lemma\;3.7(c)]{CFK} extends to \eqref{eq:ourbounds} for 
$\nu \geq 4$. Indeed, observe $K_C-p$ is very ample as it follows by the Riemann-Roch theorem;
\color{black} moreover, as in \cite[Claim\;3.8]{CFK}, for general $D \in W^0_{4g-5-d}$, $K_C-D$ is very ample if 
$\nu \ge 4$  and $d \ge 3g-1$. Here we remark that  the condition $d \ge 3g-1$ was crucially used in the proof of the claim.  \color{black}
%
%
%
Thus, as $\mathcal F \in B_{\rm reg}$ general fits in \eqref{exactB0}, part $(i)$ is proved in this case. 

When otherwise $C$ is with general moduli, $K_C-p$ is very ample. Since $\deg\; K_C-D = d-2g+3$ and it is of speciality $1$, then $h^0(K_C-D) = d-3g+5$ which is very ample as soon as the latter quantity is at least $4$.

\noindent
(ii) This part extends the proof of \cite[Lemma\;3.3(c)]{CFK} to \eqref{eq:ourbounds}. Observe first that $K_C-A$ is very ample: 
if not, Riemann-Roch theorem would give the existence of a $g^2_{\nu+2}$ on $C$, which is not allowed by Lemma \ref{S2} above.
At the same time, since $\deg(N)=d-2g+2+\nu \geq g +3$ for $d+\nu \ge 3g+1$, a general $N$ is very ample too. Thus $\mathcal F_u$ as in \eqref{exactB1} is very ample too. 
\end{proof}


\section{Hilbert schemes of surface scrolls}\label{S:Hilbert} In this section, we consider some natural consequences of Theorems \ref{thm:main2} and \ref{thm:veryampleness} 
to Hilbert schemes of surface scrolls in projective spaces. Precisely, with assumptions as in Theorem \ref{thm:veryampleness}, a general $\mathcal F \in B_{\rm reg}$ (respectively,  $\mathcal F \in B_{\rm sup}$) gives 
rise to the projective bundle $\Pp(\Ff)\stackrel{\rho}{\to} C$ ($\rho$ is the fiber-map), which is embedded via  $\vert \OO_{\Pp(\Ff)}(1)\vert$ as a smooth scroll $S$ of degree $d$, sectional genus $g$ and which is linearly normal in 
 the projective space $\Pp^{k_2-1}= \Pp^{d-2g+3}$,  as $k_2= d-2g+4$. We will say that the pair $(C, \Ff)$ {\em determines} $S$, equivalently that $S$ is {\em associated to} 
 $(C, \Ff)$.  
 
 In any of the above cases, the scroll $S$ is {\em stable}, since $\mathcal F$ is, and it is {\em special},  
 since $$h^1(S, \OO_S(1)) = h^1(\Pp(\Ff), \OO_{\Pp(\Ff)}(1)) = h^1(C,\mathcal F) =2.$$Let $P_S(t) \in \mathbb{Q}[t]$ be the Hilbert polynomial of $S$ and let 
 $$\mathcal H_{d,g,k_2-1}$$be the union of components of the Hilbert scheme, parametrizing closed 
 subschemes in $\Pp^{k_2-1}$, having Hilbert polynomial $P_S(t)$, whose general point corresponds to a smooth, linearly normal surface scroll in $\Pp^{k_2-1}$.

\begin{proposition}\label{prop:Normal} If $\N_{S/\Pp^{k_2-1}}$ denotes the normal bundle of $S$ in $\Pp^{k_2-1}$, then: 
\begin{equation}\label{eq:tgS3bis}
\chi(S,\N_{S/\Pp^{k_2-1}}) = 7g-7 + k_2(k_2-2) = 7g-7 + (d-2g+4) (d-2g+2).
\end{equation}In particular, for any irreducible component $\mathcal H$ of $\mathcal H_{d,g,k_2-1}$, one has 
$$\dim \; \mathcal H \geq \chi(S,\N_{S/\Pp^{k_2-1}}) = 7g-7 + k_2(k_2-2),$$where the latter is the so called {\em expected dimension} of the Hilbert scheme. 
\end{proposition}

\begin{proof} The {\em Euler's  sequence} restricted to $S$ is 
\begin{equation}\label{eq:EulerS}
0 \to \Oc_S \to H^0(\Oc_S(H))^{\vee} \otimes \Oc_S(H) \to \T_{\Pp^{k_2-1}}|_S \to 0.
\end{equation}Moreover, one has also the {\em normal bundle sequence}
\begin{equation}\label{eq:tang}
0 \to \T_S \to \T_{\Pp^{k_2-1}}|_S \to \N_{S/\Pp^{k_2-1}} \to 0,
\end{equation}where $\T_S$ denotes the tangent bundle of $S$. 

Since $S$ is a scroll of genus $g$, we have
\begin{equation}\label{eq:tgS}
\chi(\mathcal O_S) = 1-g, \quad  \chi(\T_S) = 6 - 6g
\end{equation}\color{black} (the latter equality is well--known. It can be easily computed: by using 
the structural scroll--morphism 
$ S \cong \mathbb{P}(\mathcal F) \stackrel{\rho}{\longrightarrow} C$ and the standard 
scroll exact sequence $0 \to \T_{rel} \to \T_S \to \rho^*(\T_C) \to 0$, where $\T_{rel}$ is the {\em relative tangent sheaf}; from the above sequence and the fact that $S$ is a scroll, one gets $\T_{rel} = \omega_S^{\vee} \otimes \rho^* (\omega_C) \cong \Oc_S(2 H - \rho^*(\det {\mathcal F})) $,  and so  $\chi(S,\T_S) = \chi (S,\T_{rel}) + \chi (S,\rho^*(\T_C)) = \chi(S, \Oc_S(2 H - \rho^*(\det {\mathcal F}) )) + \chi (C, \T_C)=2(3-3g)$). \color{black}

From  Euler's sequence above, we get
\begin{equation}\label{eq:tgS2}
\chi(\T_{\Pp^{k_2-1}}|_S) = k_2 (k_2-2) + g-1,
\end{equation} \color{black} since $\chi(S, \Oc_S(H)) = \chi (C, \mathcal F) = d-2g+2 = k_2-2$ as it follows from the fact that 
$S \cong \mathbb{P}(\mathcal F)$ is a scroll, from Leray's isomorphism and projection formula. \color{black}

Thus, from \eqref{eq:tang}, we get
$$\chi(S, \N_{S/\Pp^{k_2-1}}) =7(g-1) + k_2(k_2-2)$$as in \eqref{eq:tgS3bis}.

The last assertion \color{black} in the statement of Proposition \ref{prop:Normal} \color{black} is a consequence of \cite[Corollary 3.2.7]{Ser} and the fact that $h^2(\N_{S/\Pp^{k_2-1}})=0$, as it follows from $h^2(\Oc_S(H))=0$, \eqref{eq:EulerS} and \eqref{eq:tang}. 
\end{proof}

With this set--up, the aim of this section is to prove Theorem \ref{thm:Hilb}. This will be done in the following subsections.


\subsection{The components $\mathcal H_{{\rm sup},\nu}$'s}\label{ss:Hsup} In the following section we will give the proof of Theorem \ref{thm:Hilb} (ii). 
We start giving a parametric construction of the components $\mathcal H_{{\rm sup},\nu}$'s for every possible $(d,g,\nu)$ arising from  
\eqref{eq:ourbounds} and conditions in Theorems \ref{thm:Hilb} (ii) and \ref{thm:veryampleness}\;(ii).

To this aim, consider:  
\begin{itemize}
\item $C \in \mathcal M^1_{g,\nu}$ general
\item $\Ff \in B_{\rm sup}$ general on $C$
\item $\Phi \in {\rm PGL}(k_2, \mathbb{C}) = {\rm Aut}(\Pp^{k_2-1})$.
\end{itemize}The triple $(C, \Ff, \Phi)$ determines the smooth scroll $\Phi(S) \subset \Pp^{k_2-1}$, where $S$ is associated to $(C, \Ff)$. 

For each triple $(d,g,\nu)$, scrolls $\Phi(S)$ as above fill--up an irreducible subset $\mathcal X_{\nu}$ of $\mathcal H_{d,g,k_2-1}$, as $\mathcal M^1_{g,\nu}$, $B_{\rm sup}$ on $C$ and  ${\rm PGL}(k_2, \mathbb{C})$ are all 
irreducible. Therefore, $\mathcal X_{\nu}$ is contained in (at least) one irreducible component of $\mathcal H_{d,g,k_2-1}$; any such 
irreducible component dominates $\mathcal M^1_{g,\nu}$ (as $\mathcal X_{\nu}$ does, by construction) and has dimension at least $\dim\;\mathcal X_{\nu}$. 

Thanks to the parametric representation of $\mathcal X_{\nu}$, we can easily compute its dimension.

\begin{proposition}\label{cl:dimX} $\dim\; \mathcal X_{\nu} = 8g-d-12 + k_2^2 = 8g-d - 12 + (d - 2g + 4)^2.$ In particular, if  $d\leq 3g-4$, then the closure of $\mathcal X_{\nu} $ cannot be an irreducible component of $\mathcal H_ {d,g,k_2 -1}$. 
\end{proposition}

\begin{proof} Let $(C, \Ff, \Phi)$ and $(C', \Ff', \Phi')$ be two triples such that $\Phi(S) = \Phi'(S')$. Since 
$\Phi$ and $\Phi'$ are both projective transformations, the previous equality implies $S' = ((\Phi')^{-1} \circ \Phi)(S)$, i.e. $S$ and $S'$ are projectively equivalent via $\Psi:= ((\Phi')^{-1} \circ \Phi)$ and the 
triples $(C, \Ff, {\rm Id})$  $(C', \Ff', \Psi)$ map to the same point in $\mathcal X_{\nu}$. 

This, in particular, implies that the abstract ruled surfaces $\Pp(\Ff)$ on $C$ and $\Pp(\Ff')$ on $C'$ are isomorphic via $\Psi$. Thus, 
$\Psi|_C: C \to C'$ has to be an isomorphism, i.e. $C$ and $C'$ corresponds to the same point of  $\mathcal M^1_{g,\nu}$ and $\Psi|_C \in {\rm Aut}(C)$. 
On the other hand, since $C \in \mathcal M^1_{g,\nu}$ is general, with $\nu \geq 3$, one has ${\rm Aut}(C) = \{Id_C\}$ (cf.\;computations in\;\cite[pp.\;275-276]{GH}). 
Therefore, with notation as in \cite{Ma}, $\Psi \in  {\rm Aut}_C(\Pp(\Ff))$, which is the subgroup of ${\rm Aut}(\Pp(\Ff))$ of automorphisms of $\Pp(\Ff)$ over $C$ 
(i.e. fixing $C$ pointwise). 

From \cite[Lemma\;3]{Ma}, one has the exact sequence of algebraic groups
$$\{Id\} \to \frac{{\rm Aut}(\Ff)}{\mathbb C^*} \to {\rm Aut}_C(\Pp(\Ff)) \to \Delta \to \{Id\}$$where $\Delta$ is a finite subgroup of the $2$-torsion part of ${\rm Pic}^0(C)$. 
Since $\Ff$ is stable, so simple (i.e. ${\rm Aut}(\Ff) \cong \mathbb C^*$), we deduce that ${\rm Aut}_C(\Pp(\Ff))$ is a finite group. 

This means that 
$$\dim\; \mathcal X_{\nu}= \dim \; \mathcal M^1_{g,\nu} + \dim \; B_{\rm sup} + \dim\; {\rm PGL}(k_2, \mathbb{C});$$the latter sum is: 
$$(2g+2\nu - 5) + (6g-d-2\nu -6) + (k_2^2-1) = 8g - d -12 + k_2^2 = 8g-d - 12 + (d - 2g + 4)^2.$$Notice moreover that, if $d=3g-5, \; 3g-4$, one has that 
$$\dim\; \mathcal X_{\nu} < \chi(\mathcal N_{S/\Pp^{k_2-1}}) = 7g-7 + k_2(k_2-2)$$as in \eqref{eq:tgS3bis}; this means that, in these cases, $\mathcal X_{\nu}$ cannot be a component of the Hilbert 
scheme as it follows by Proposition \ref{prop:Normal}.
\end{proof}

Let $[S] \in \mathcal X_{\nu} \subset \mathcal H_{d,g,k_2-1}$ be the point corresponding to the scroll $S \subset \Pp^{k_2-1}$; then $$T_{[S]} (\mathcal H_{d,g,k_2-1}) \cong H^0(S,\mathcal N_{S/\Pp^{k_2-1}}).$$

We first focus on the case  $d \geq 3g-3$ and prove that $\mathcal X_{\nu}$ fills--up a  component of $\mathcal H_{d,g,k_2-1}$ with properties as in Theorem \ref{thm:Hilb} (ii). To prove this, 
we are reduced to comput\color{black}ing \color{black} the cohomology of $\mathcal N_{S/\Pp^{k_2-1}}$ for $[S]$ a general point of $\mathcal X_{\nu}$. This will be done in the following proposition.

\begin{proposition}\label{prop:Normalsup} Let $S \subset \Pp^{k_2-1}$ be a smooth, linearly normal, special scroll which corresponds to a general 
point of  $\mathcal X _\nu$ as above for $d\geq 3g-3$. Then, one has:
\begin{itemize}
\item[(i)] $h^0( S, \N_{S/\Pp^{k_2-1}}) = 8g-d-12 + k_2^2 = 8g-d-12+(d-2g+4)^2$;
\item[(ii)] $h^1( S, \N_{S/\Pp^{k_2-1}}) = d - 3g + 3 $;
 \item[(iii)] $h^2( S, \N_{S/\Pp^{k_2-1}}) = 0$.
\end{itemize}
\end{proposition}

\begin{proof}  Observe that (iii) has already been proved in Proposition \ref{prop:Normal}. We moreover observed therein that 
$$\chi(S, \N_{S/\Pp^{k_2-1}}) = h^0(S, \N_{S/\Pp^{k_2-1}}) - h^1(S, \N_{S/\Pp^{k_2-1}})=
7(g-1) + k_2(k_2-2)$$as in \eqref{eq:tgS3bis}. Therefore, the rest of the proof is concentrated on  computing $h^1(S, \N_{S\vert \Pp^{k_2-1}})$.

Since $S \cong \Pp(\Ff)$ is a scroll corresponding to a general point of $\mathcal X$, then $\mathcal F$ corresponds to the general point of $B_{\rm sup}$ on $C$. 
Let $\Gamma$ be the unisecant of $S$ of degree $2g-2-\nu$ corresponding to \color{black} the \color{black} quotient line bundle $\Ff \to\!\!\!\!\! \to K_C-A$ as in \eqref{exactB1} (cf.\;\cite[Ch.\;V.\;Proposition\;2.6]{H}). 

\begin{claim}\label{cl:flam1502} One has $h^1(S, \N_{S /\Pp^{k_2-1}} (-\Gamma)) = h^2(S, \N_{S /\Pp^{k_2-1}} (-\Gamma))  = 0$, hence
\begin{equation}\label{eq:tgS15}
h^1(S, \N_{S /\Pp^{k_2-1}}) = h^1(\Gamma, \N_{S /\Pp^{k_2-1}}|_{\Gamma}).
\end{equation}
\end{claim}

\begin{proof}[Proof of Claim \ref{cl:flam1502}] Look at the exact
sequence
\[0 \to \N_{S /\Pp^{k_2-1}} (-\Gamma) \to  \N_{S /\Pp^{k_2-1}} \to
\N_{S /\Pp^{k_2-1}}|_{\Gamma} \to 0.\]
From \eqref{eq:tang} tensored by $\Oc_S(-\Gamma)$ we see that $h^2(S, \N_{S /\Pp^{k_2-1}} (-\Gamma)) = 0$
follows from $h^2(S,  \T_{\Pp^r}|_S (-\Gamma)) = 0$ which, by 
Euler's sequence restricted to $S$, follows from $h^2(S, \Oc_S(H - \Gamma)) =
h^0(S, \Oc_S(K_S - H + \Gamma)) = 0$, since $K_S - H + \Gamma$ intersects the ruling of $S$ negatively.

As for $h^1(S, \N_{S /\Pp^{k_2-1}} (-\Gamma))= 0$, this follows from $h^1(S,
\T_{\Pp^{k_2-1}}|_S (-\Gamma)) = h^2(S,  \T_S (-\Gamma)) = 0$. By  Euler's
sequence restricted to $S$, the first vanishing follows from $h^2(S, \Oc_S(-\Gamma)) =
h^1(S, \Oc_S(H-\Gamma))=0$. Since $K_S+\Gamma$ meets the ruling
negatively, one has $h^0(S, \Oc_S(K_S +\Gamma)) = h^2(S, \Oc_S(-\Gamma))
=0$. Moreover $h^1(S, \Oc_S(H-\Gamma)) = h^1(C, N)=0$, as it follows from 
\eqref{exactB1} and the fact that $N \in {\rm Pic}^{d-2g+2+\nu}(C)$ is non special, being general of its degree 
(cf. Theorem \ref{thm:main2}\;(ii)).

In order to prove $h^2(S,  \T_S (-\Gamma)) = 0$, consider the exact
sequence
\[0 \to \T_{rel} \to \T_S \to \rho^*(\T_C) \to 0\]arising from the structure morphism
$S\cong \Pp(\Ff) \stackrel{\rho}{\to} C$. The vanishing we need follows from 
$h^2(S,  \T_{rel} \otimes \Oc_S(-\Gamma)) = h^2 (S, \Oc_S(-\Gamma) \otimes \rho^*(\T_C)) = 0$: 
the first vanishing holds since $\T_{rel} \cong \Oc_S (2H - df)$, where $f = \rho^{-1}(q)$ is a ruling of $S$, therefore 
$\Oc_S(K_S + \Gamma) \otimes \T_{rel}^{*}$ restricts negatively to the ruling, so it cannot be effective. 
Similar considerations  yield the second vanishing $h^2 (S, \Oc_S(-\Gamma) \otimes \rho^*(\T_C)) = 0$.  
\end{proof}

Consider now the exact sequence
\begin{equation}\label{eq:B}
0 \to  \N_{\Gamma/S}  \stackrel{\alpha}{\longrightarrow}  \N_{\Gamma /\Pp^{k_2-1}} {\longrightarrow}
\N_{S / \Pp^{k_2-1}}|_{\Gamma}  \to 0.
\end{equation}

\begin{claim}\label{cl:flam2611} The map 
$$H^1(\Gamma, \N_{\Gamma/S}) \stackrel{H^1(\alpha)}{\longrightarrow}
H^1(\Gamma, \N_{\Gamma/\Pp^{k_2-1}})$$arising from \eqref{eq:B} is injective. 
\end{claim}

\begin{proof}[Proof of Claim \ref{cl:flam2611}] Consider $\Gamma \subset \langle \Gamma \rangle = \Pp^{g-\nu} \subset \Pp^{k_2-1}$,
where $\langle \Gamma \rangle$ denotes the linear span given by the section $\Gamma$ and where $\dim \; \langle \Gamma \rangle  := h^0(K_C-A) - 1 = h^1(A) - 1 = g-\nu$, as it follows from \eqref{eq:h12A}.  

From the inclusions $\Gamma \subset \Pp^{g-\nu} \subset \Pp^{k_2-1}$ we get the sequence
\begin{equation}\label{eq:tg*}
0 \to \N_{\Gamma/\Pp^{g-\nu}} \to \N_{\Gamma\vert \Pp^{k_2-1}} \to
\N_{\Pp^{g-\nu}/\Pp^{k_2-1}}|_{\Gamma} \to 0,
\end{equation}Take the Euler sequence of $\Pp^{g-\nu}$ restricted to $\Gamma$, i.e. 
$$0 \to \Oc_{\Gamma} \to  H^0(\Oc_{\Gamma}(1))^{\vee} \otimes \Oc_{\Gamma}(1) \cong (K_C-A)^{\oplus(g-\nu+1)} \to \T_{\Pp^{g-\nu}}|_{\Gamma} \to 0;$$ \color{black} taking cohomology and dualizing, \color{black} we get that 
$$H^1(\T_{\Pp^{g-\nu}}|_{\Gamma})^{\vee} \cong {\rm Ker}\left(H^0(K_C-A) \otimes H^0(A) \stackrel{\mu_{0,A}}{\longrightarrow} H^0(K_C)\right)$$as in \eqref{eq:kers}. Therefore, from \color{black}
\eqref{eq:kers} and \color{black} \eqref{eq:h12A} one gets
$$h^1(\T_{\Pp^{g-\nu}}|_{\Gamma}) = g + 2 - 2 \nu.$$Consider now the exact sequence defining the normal bundle of $\Gamma$ in its linear span: 
$$0 \to \T_{\Gamma} {\longrightarrow} \T_{\Pp^{g-\nu}}|_{\Gamma} {\longrightarrow} \N_{\Gamma/\Pp^{g-\nu}} \to 0;$$the associated coboundary map 
$H^0( \N_{\Gamma/\Pp^{g-\nu}}) \stackrel{\partial}{\longrightarrow} H^1 (\T_{\Gamma})$ identifies with the differential at the point $[\Gamma]$ of the natural map 
$$\Psi: {\rm Hilb}_{g,2g-2-\nu, g-\nu} \to \mathcal M_g,$$where ${\rm Hilb}_{g,2g-2-\nu, g-\nu}$ the Hilbert scheme of curves of genus $g$, degree $2g-2-\nu$ in $\Pp^{g-\nu}$. By construction, 
$$\dim\; {\rm coker}(d\Psi_{[\Gamma]}) = \dim \; \mathcal M_g - \dim \;\mathcal M^1_{g,\nu} = 3g-3 - (2g+2\nu -5) = g+2 - 2\nu = h^1(\T_{\Pp^{g-\nu}}|_{\Gamma}),$$i.e. 
the map 
$$H^1( \T_{\Gamma}) \stackrel{H^1(\lambda)}{\longrightarrow} H^1(\T_{\Pp^{g-\nu}}|_{\Gamma})$$is surjective. Since $h^2(\T_{\Gamma}) = 0$, this implies $h^1(\N_{\Gamma/\Pp^{g-\nu}}) = h^2(\N_{\Gamma/\Pp^{g-\nu}}) = 0$. Therefore, 
from \eqref{eq:tg*}, one has 
\begin{equation}\label{eq:normals1}
H^1(\N_{\Gamma/\Pp^{k_2-1}}) \cong H^1(\N_{\Pp^{g-\nu}/\Pp^{k_2-1}}|_{\Gamma}) = H^1(\Oc_{\Gamma}(1)^{\oplus(k_2-1 - g + \nu)}) \cong H^1((K_C-A)^{\oplus(k_2-1 - g + \nu)}).
\end{equation}

Since the scroll $S$ arises from $\Ff \in B_{\rm sup}$ general (on $C \in \mathcal M^1_{g,\nu}$ general), then $\Ff$ fits in \eqref{exactB1}, with $N$ general of its degree. In particular one has 
$$0 \to H^0(N) \to H^0(\Ff) \to H^0(K_C-A) \to 0.$$Therefore, one has also 
$$0\to H^0(C, K_C-A)^{\vee} \to H^0(C, \Ff)^{\vee} \to H^0(C, N)^{\vee} \to
0.$$Since $H^0(S, \Oc_S(1)) \cong H^0(C, \Ff)$ and $\Oc_{\Gamma}(1)\cong K_C-A$, the Euler sequences of the projective spaces 
$\Pp^{g-\nu}$ and $\Pp^{k_2-1}$ restricted to $\Gamma$ give the following commutative diagram:
\begin{displaymath}
\begin{array}{ccccccc}
      &             &     &     0                       &     &     0     &     \\
       &             &     &     \downarrow                        &     &     \downarrow      &     \\
0 \to & \Oc_{\Gamma} & \to & H^0(C, K_C-A)^{\vee} \otimes K_C-A
& \to & \T_{\Pp^{g-\nu}}|_{\Gamma} & \to 0  \\
      &      ||       &     &     \downarrow                        &     &     \downarrow      &     \\
0 \to & \Oc_{\Gamma} & \to & H^0(C, \Ff)^{\vee} \otimes
K_C-A & \to & \T_{\Pp^{k_2-1}}|_{\Gamma} & \to 0  \\
     &             &     &     \downarrow                        &     &     \downarrow      &     \\
 &  &  & H^0(C, N)^{\vee} \otimes K_C-A&\cong  & \N_{\Pp^{g-\nu}\vert \Pp^{k_2-1}}|_{\Gamma} & \\
          &             &     &     \downarrow                        &     &     \downarrow      &     \\
     &             &     &     0                       &     &     0     &
\end{array}
\end{displaymath}This shows in particular that 
$$\N_{\Pp^{g-\nu}\vert \Pp^{k_2-1}}|_{\Gamma} \cong H^0(C, N)^{\vee} \otimes K_C-A$$and so, in \eqref{eq:normals1}, one has more precisely 
$$H^1(\N_{\Gamma/\Pp^{k_2-1}}) \cong H^1(\N_{\Pp^{g-\nu}/\Pp^{k_2-1}}|_{\Gamma}) \cong H^0(N)^{\vee} \otimes h^1(K_C-A).$$

On the other hand, $\N_{\Gamma/S} \cong K_C-A-N$ (cf. \cite[Ch.\;V.\;Proposition\;2.6]{H}), so 
$$h^0(\Gamma, \N_{\Gamma/S})= 0, \; h^1(\Gamma, \N_{\Gamma/S}) = d- 3g+3 + 2 \nu.$$Therefore, if we take the map 
$$H^1(\Gamma, \N_{\Gamma/S}) \stackrel{H^1(\alpha)}{\longrightarrow} H^1(\Gamma, \N_{\Gamma/\Pp^{k_2-1}})$$arising from \eqref{eq:B}, this identifies with 
the natural map 
$$H^1(K_C-A-N) \stackrel{H^1(\alpha)}{\longrightarrow}  H^0(N)^{\vee} \otimes H^1(K_C-A)$$whose dual is 
$$H^0(N) \otimes H^0(A) \stackrel{H^1(\alpha)^{\vee}}{\longrightarrow}  H^0(N+A),$$i.e. $H^1(\alpha)^{\vee} = \mu_{A,N}$ as in \eqref{muA} is a natural multiplication map. 
Since $N$ is non special and by definition of $A$, one has 
$$h^0(N) = d - 3g + 3 + \nu, \; h^0(A) = 2;$$moreover, from \eqref{eq:kers}, one has 
$${\rm ker} (\mu_{A,N}) = h^0(N-A) = d - 3g + 3.$$Therefore, 
$$\dim {\rm Im}\;\mu_{A,N} = 2 (d - 3g + 3 + \nu) - \color{black} (\color{black}d - 3g + 3 \color{black}) \color{black}= d - 3g + 3 + 2 \nu = h^0(N+A),$$i.e. $\mu_{A,N} = H^1(\alpha)^{\vee}$ is surjective. This implies that 
$H^1(\alpha)$ in injective, as wanted. 
\end{proof}Considering once again \eqref{eq:B} and \eqref{eq:normals1}, the injectivity of $H^1(\alpha)$ and $h^2(\N_{\Gamma/S}) =0$ give 
$$h^1(\N_{S / \Pp^{k_2-1}}|_{\Gamma}) = h^1(\N_{\Gamma /\Pp^{k_2-1}}) - h^1(\N_{\Gamma/S}) = $$
$$2 (k_2 - 1 + g - \nu) - h^1(K_C-A-N) = 2 ((k_2 - 1 + g - \nu)) - (d + 2 \nu - 3 g + 3) = d - 3g + 3.$$From \eqref{eq:tgS15}, the (ii) of Proposition \ref{prop:Normalsup} 
follows and the proof is completed.
\end{proof}

\color{black}

To conclude the proof of Theorem \ref{thm:Hilb} (ii), for $d \geq 3g-3$, we need to show that $\mathcal X_{\nu}$ 
fills--up a dense subset of a unique component, say $\mathcal H_{{\rm sup},\nu}$, with all the properties mentioned therein.  
To deduce this, it suffices to observe first that 
$$ 8g-d-12 + k_2^2 = \dim\; \mathcal X_{\nu} \leq \dim \; T_{[S]} (\mathcal H_{d,g,k_2-1}) = h^0(S,\N_{S / \Pp^{k_2-1}})=  8g-d-12 + k_2^2,$$the latter equality following from 
Proposition \ref{prop:Normalsup} (i). Moreover, as 
$$8g-d-12 + k_2^2 = \chi(\N_{S / \Pp^{k_2-1}})+ d-3g+3,$$ it follows that the component  $\mathcal H_{{\rm sup},\nu}$ (arising as the closure of $\mathcal X_{\nu}$ in $\mathcal H_{d,g,k_2-1}$) 
is a {\em superabundant} (resp., {\em regular}) component of $\mathcal H_{d,g,k_2-1}$ for $d\geq 3g-2$ (resp. $d=3g-3$). By construction of $\mathcal H_{{\rm sup},\nu}$, it follows that it dominates 
$\mathcal M^1_{g,\nu}$. This implies that $\mathcal H_{{\rm sup},\nu} \neq \mathcal H_{{\rm sup},\nu'}$ for $\nu \neq \nu'$. Thus the proof of  Theorem \ref{thm:Hilb} (ii) is completed for $d \geq 3g-3$.

Now we take into account the cases $3g-5 \leq d \leq 3g-4$; recall that $\mathcal X_{\nu}$ has to be strictly contained in at least one irreducible component $\mathcal H$ of $\mathcal H_{d,g,k_2-1}$. To investigate such a component $\mathcal H$, 
we will use the following lemma. 

\begin{lemma}\label{lem_special} For $3g-5 \leq d \leq 3g-4$, assume that $\mathcal H$ is an irreducible component of $\mathcal H _{d,g, k_2 -1}$, whose general point corresponds to a smooth, stable scroll. 
Let  $\mathcal F _u$ be a rank 2 vector bundle associated to a general element of  $\mathcal H$, where $\mathcal F_u$ arises as an extension of the form 
\eqref{degree}, with $L$ necessarily special, on a suitable smooth curve $C$ of genus $g$. Then, one must have $h^1(L) =1$. 
 \end{lemma}

\begin{proof} By the definition of $\mathcal H _{d,g, k_2 -1}$, the general point $[S]$ of $\mathcal H$ represents a smooth, linearly normal scroll $S$ in 
$\Pp^{k_2-1}$, i.e. it is of speciality exactly $2$; the scroll $S$ is associated to a degree $d$, very ample, rank 2 vector bundle $\mathcal F_u$ on a smooth curve $C$ of genus $g$. With a small abuse of notation, in what follows we will denote simply 
by $u \in \mathcal H$ the corresponding point $[S]$. 

From the fact that $\mathcal F_u$ is special and stable, by Theorem \ref{CF}  $\mathcal F_u$ arises as an extension \eqref{degree}.  
Suppose that $h^1 (L) > 1$, then one must have $h^1(L) =2$. Since $\mathcal F_u$ is stable with $d \geq 3g-5$ then, by \eqref{eq:neccond}, one has $\delta:= \deg\; L > \frac{3g-5}{2}$. Then 
$|K_C -L|$ is a $g^1_k$ with $k < \frac{g+1}{2}$,  where $k:=2g-2-\deg \; L$. 

Thus there exists an open dense subset $\mathcal H_0$ of $\mathcal H$ which admits a map:
$$\eta : \mathcal H_0 \to {\mathcal P}ic ^k (p)$$  given by $\eta (u) := K_C -L$,
where ${\mathcal P}ic ^k (p)$ is the relative Picard variety for  $p: \mathcal C \to S$ a suitable family of smooth curves of genus $g$.

By $h^1(L) =2$, the image of $\eta$ is included in $\mathcal W ^1_k$, where $\mathcal W ^1_k$ is a subvariety of ${\mathcal P}ic ^k (p)$  parameterizing pairs $(C, M)$ with $h^0 (M) \geq 2$.  
It is known that $dim \; \mathcal W ^1_k =2g +2k -5$ for $k < \frac{g+1}{2}$ (see \cite[Proposition (6.8)]{ACG}). The fiber of $\eta$ has dimension at most
$$\dim \; {\rm Pic}^{d-\delta} (C)+ \dim \; \mathbb P (\ext^1(L, N)) + \dim \;{\rm PGL}(k_2, \mathbb{C}) = 6g -d-2k-6+k_2 ^2 -1$$as it follows by \eqref{eq:m}. 
In sum, we get:  
$$\dim \; \mathcal H \leq (2g +2k-5) + (6g -d-2k-6 +k_2 ^2 -1) =8g-d-12 +k_2 ^2.$$
This cannot occur for $d\leq 3g-4$, since any irreducible component has dimension at least 
$\chi (\N_{S / \Pp^{k_2-1}})$ as in Proposition \ref{prop:Normal}.
\end{proof}

\begin{corollary}
If $d=3g-5, \; 3g-4$, $\mathcal X_{\nu}$ is strictly contained a component of $\mathcal H_{d,g,k_2-1}$ whose general point is associated to an extention \eqref{degree} with $h^1 (L)=1$ on a suitable smooth curve of genus $g$. 
\end{corollary}

\subsection{The component $\mathcal H_{\rm reg}$}\label{ss:Hreg} As we did above for the components $\mathcal H_{{\rm sup},\nu}$'s, we first give a parametric construction of the component 
$\mathcal H_{\rm reg}$.

Take integers $d,\;g,\;\nu$ as in \eqref{eq:ourbounds} and in Theorem \ref{thm:veryampleness}\;(i). As observed therein, the construction of $B_{\rm reg}$, conducted in Sect.\;\ref{ss:regular} for $C \in \mathcal M^1_{g,\nu}$ general, 
holds {\em verbatim} for $C$ with general moduli and, in particular, it coincides with the (unique) component $\mathcal B$ of $B_d^{k_2}\cap U_C(2,d)^s$ as in the statement of 
Theorem \ref{TeixidorRes}; moreover, very-ampleness conditions in Theorem \ref{thm:veryampleness}\;(i) holds also for $C$ general. 

To construct $\mathcal H_{\rm reg}$, take therefore:    
\begin{itemize}
\item $C \in \mathcal M_g$ general
\item $\Ff \in B_{\rm reg}$ general on $C$
\item $\Phi \in {\rm PGL}(k_2, \mathbb{C}) = {\rm Aut}(\Pp^{k_2-1})$.
\end{itemize}As in the previous section, the triple $(C, \Ff, \Phi)$ determines the smooth scroll $\Phi(S) \subset \Pp^{k_2-1}$, where 
$S$ is associated to $(C, \Ff)$. Such scrolls $\Phi(S)$ fill-up an irreducible subset $\mathcal Y$ of $\mathcal H_{d,g,k_2-1}$, as $\mathcal M_g$, $B_{\rm reg}$ on $C$ and  ${\rm PGL}(k_2, \mathbb{C})$ are all 
irreducible. Therefore, $\mathcal Y$ is contained in (at least) one irreducible component of $\mathcal H_{d,g,k_2-1}$; any such 
component dominates $\mathcal M_g$ (as $\mathcal Y$ does, by construction) and it is of  
dimension at least $\dim\;\mathcal Y$. Moreover, since $ \mathcal M^1_{g,\nu} \subsetneq \mathcal M_g$ is also irreducible and, for $C' \in  \mathcal M^1_{g,\nu}$ 
general, $B_{\rm reg}$ on $C'$ is irreducible, of the same dimension as $B_{\rm reg}$ on $C$, the triples $(C', \Ff', \Phi)$, for $C' \in  \mathcal M^1_{g,\nu}$ 
general, $\Ff' \in B_{\rm reg}$ general on $C'$ and $\Phi \in {\rm PGL}(k_2, \mathbb{C})$ fill--up an irreducible, closed subset $\mathcal Y' \subsetneq \mathcal Y$, 
where $\mathcal Y'$ dominates $ \mathcal M^1_{g,\nu}$ (but not $\mathcal M_g$) by construction. 

Thanks to the parametric representation of $\mathcal Y$, reasoning as in the proof of Proposition \ref{cl:dimX}, one can easily compute $\dim\;\mathcal Y$, as 
${\rm Aut}(C) = \{Id_C\}$ for $C$ with general moduli. Thus, one gets: 
$$\dim\; \mathcal Y = \dim \; \mathcal M_{g} + \dim \; B_{\rm reg} + \dim\; {\rm PGL}(k_2, \mathbb{C});$$the latter quantity is 
$$(3g-3) + (8g-2d-11) + (k_2^2-1) = 11g - 2d -15 + k_2^2 = 11g - 2d - 15 + (d-2g+4)^2 = $$
$$= 11g - 2d - 15 + 2 (d-2g+4) + (d-2g+2)(d-2g+4)= 7g - 7 +(d-2g+4) (d-2g+2) = 7g - 7 + k_2 (k_2-2).$$

To prove that $\mathcal Y$ fills--up a dense subset of a unique component of $\mathcal H_{d,g,k_2-1}$, with properties as in Theorem \ref{thm:Hilb} (i), we are reduced to compute the cohomology of the normal bundle 
$\mathcal N_{S/\Pp^{k_2-1}}$ for $S$ corresponding to a general point of $\mathcal Y$. This will be done in the following:

\begin{proposition}\label{prop:Normalreg} Let $S \subset \Pp^{k_2-1}$ correspond to a general point of  
$\mathcal Y$ as above. Then, one has: 
\begin{itemize}
\item[(i)] $h^0( S, \N_{S/\Pp^{k_2-1}}) = 7g-7 + k_2(k_2-2) = 7g-7 + (d-2g+4) (d-2g+2)$;
\item[(ii)] $h^1( S, \N_{S/\Pp^{k_2-1}}) = 0$;
 \item[(iii)] $h^2( S, \N_{S/\Pp^{k_2-1}}) = 0$.
\end{itemize}
\end{proposition}

\begin{proof} The proof of (iii) has already been given in Proposition \ref{prop:Normal}. Therefore, 
$\chi(\N_{S/\Pp^{k_2-1}}) = h^0(\N_{S/\Pp^{k_2-1}}) - h^1(\N_{S/\Pp^{k_2-1}} ) $ \color{black} is given in \color{black} \eqref{eq:tgS3bis}. \color{black} The proof is reduced to showing that \color{black} $h^1(S, \N_{S\vert \Pp^{k_2-1}}) =0$.

Since $S \cong \Pp(\Ff)$ corresponds to a general point of $\mathcal Y$, then $\mathcal F$ corresponds to a general point of $B_{\rm reg}$ on 
$C$ with general moduli.  To compute $h^1(S, \N_{S/\Pp^{k_2-1}})$, we therefore cannot proceed as in the proof of Proposition \ref{prop:Normalsup} 
(where we used the section of minimal degree $\Gamma$ corresponding to the quotient line bundle $K_C-A$ and the fact that $h^1(C,N) =0$ for $N$ as in \eqref{exactB1}). 
Indeed,  in the present case the section $\Gamma$ corresponds to the quotient line bundle $\Ff \to\!\!\!\!\! \to K_C-p$ as in \eqref{exactB0}, for which 
$h^1(K_C-D) = h^1(K_C-p) =1$. To sum--up, one cannot reason as in the previous case.

To this aim, consider the natural exact sequence on $S$: 
\begin{equation}\label{tgrel}
0 \to \T_{rel} \to \T_S \to \rho^*(\T_C) \to 0,
\end{equation}arising from the structure morphism $S\cong \Pp(\Ff) \stackrel{\rho}{\to} C$.

One has $h^2(\T_S) =0$, as it follows from  
\begin{equation}\label{tgrel2}
0 \to \rho_* (\T_{rel}) \to \rho_*(\T_S) \to \T_C \to 0,
\end{equation}obtained by push-forword \eqref{tgrel} on $C$, 
and from Leray's isomorphisms.

From the exact sequence defining the normal bundle: 
\begin{equation}\label{eq:normal}
0 \to \T_S \stackrel{\gamma_S}{\longrightarrow} \T_{\Pp^{k_2-1}}|_S  \to \N_{S/\Pp^{k_2-1}} \to 0
\end{equation}and the fact that $h^2(\T_S) =0$, one has: 
$$(*)\;\;\; h^1(\N_{S/\Pp^{k_2-1}}) =0 \;\; \Leftrightarrow \;\; H^1(\T_S) \stackrel{H^1(\gamma_S)}{\longrightarrow} H^1(\T_{\Pp^{k_2-1}}|_S) \;\; \mbox{is surjective};$$therefore, we are reduced 
to show\color{black}ing \color{black} that the map $H^1(\gamma_S)$ is a surjective map.

On the other hand, since \eqref{tgrel}, \eqref{tgrel2} and Leray's isomorphisms give 
$h^0(\rho^*(\T_C) ) = h^0(\T_C) = h^2(\T_{rel}) = h^2(\rho_* (\T_{rel}))= 0$, then one has  
$$H^1(\T_S) \cong H^1(\rho_*(\T_S)) = H^1(\rho_* (\T_{rel})) \oplus H^1(\T_C);$$moreover, from 
$K_S = - 2 H \otimes \rho^*(\omega_C \otimes \det (\mathcal F))$ (cf. \cite[Ch.\;V]{H}), one 
gets $\T_{rel} \cong \Oc_S ( 2 H \otimes \rho^*(\det (\mathcal F)^*))$ thus, by projection formula, 
$\rho_*(\T_{rel}) = Sym^2(\mathcal F) \otimes \det (\mathcal F)^*$. To sum up, one has: 
\begin{equation}\label{eq:isom1}
H^1(S, \T_S) \cong H^1\left(C, Sym^2(\mathcal F) \otimes \det (\mathcal F)^*\right) \oplus H^1\left(C, \T_C\right). 
\end{equation} 

Similarly, the Euler sequence of $\mathbb{P}^{k_2-1}$ restricted to $S$ reads:  
\begin{equation}\label{eq:EulerSS}
0 \to \Oc_S \to H^0(\mathcal F)^{\vee} \otimes \Oc_S(H) \stackrel{\tau_S}{\longrightarrow} \T_{\Pp^{k_2-1}}|_S \to 0,
\end{equation}as it follows by the definition of $\Oc_S(H)$ and the fact that $S \subset \mathbb{P}^{k_2-1}$ is linearly normal. 
Applying $\rho_*$ to \eqref{eq:EulerSS}, one has: 
\begin{equation}\label{eq:EulerC}
0 \to \Oc_C \to H^0(\mathcal F)^{\vee} \otimes \mathcal F \stackrel{\rho_*(\tau_S)}{\longrightarrow} \rho_*(\T_{\Pp^{k_2-1}}|_S) \to 0,
\end{equation}with $H^i(S, \T_{\Pp^{k_2-1}}|_S) \cong H^i(C, \rho_*(\T_{\Pp^{k_2-1}}|_S))$, for $i \geq 0$. 

Since the above identifications have been all obtained by using \eqref{tgrel} and \eqref{eq:EulerSS}, which are both compatible with 
\eqref{eq:normal}, then one has: 
{\small 
$$(**)\;\;\; h^1(\N_{S/\Pp^{k_2-1}}) =0 \;\; \Leftrightarrow \;\; 
H^1\left(C, Sym^2(\mathcal F) \otimes \det (\mathcal F)^*\right) \oplus H^1\left(C, \T_C\right) \stackrel{H^1(\rho_*(\gamma_S))}{\longrightarrow} H^1(\rho_*(\T_{\Pp^{k_2-1}}|_S)) \to 0.$$
}

From \eqref{eq:EulerSS} and $h^2(\Oc_S)=0$, one has 
$$H^0(\mathcal F)^{\vee} \otimes H^1(\Oc_S(H)) \stackrel{H^1(\tau_S)}{\longrightarrow} H^1(\T_{\Pp^{k_2-1}}|_S) \to 0$$and, as above 
$H^1(\tau_S)$ identifies with the surjective map
\begin{equation}\label{eq:surjectivity1}
H^0(\mathcal F)^{\vee} \otimes H^1(\mathcal F) \stackrel{\tiny H^1(\rho_*(\tau_S))}{\longrightarrow} H^1(\rho_*(\T_{\Pp^{k_2-1}}|_S)) \to 0
\end{equation} Therefore, to show the surjectivity of $H^1(\rho_*(\gamma_S))$ as in $(**)$, it suffices to show there exists a natural surjective map 
\begin{equation}\label{eq:surjectivity2}
H^1\left(C, Sym^2(\mathcal F) \otimes \det (\mathcal F)^*\right) \oplus H^1\left(C, \T_C\right)  
\stackrel{\psi_C}{\longrightarrow}  H^0(\mathcal F)^{\vee} \otimes H^1(\mathcal F)         
\end{equation} compatible with the maps in the previous diagrams. 

By duality, this is equivalent to prove the existence of an injective map 
\begin{equation}\label{eq:injectivity1}
H^0(\mathcal F) \otimes H^0(\omega_C \otimes \mathcal F^*) \stackrel{\psi^{\vee}_C}{\hookrightarrow}
H^0\left(C, \omega_C \otimes Sym^2(\mathcal F^*) \otimes \det (\mathcal F) \right) \oplus H^0\left(C, \omega_C^{\otimes 2}\right)
\end{equation} compatible with the dual maps of the previous diagrams.  

Since $\mathcal F$ fits in an exact sequence of the form \eqref{exactB0}, for $p$ and $D$ general on $C$ a curve with general moduli, i.e. 
$\mathcal F = \mathcal F_u$ for $u \in \mathcal W_1 \subsetneq \ext^1(K_C-p, K_C-D)$, by semicontinuity on 
$\mathcal W_1$ and the fact that $$H^0(\mathcal F_u) \cong H^0(K_C-D) \oplus H^0(K_C-p) \;\;\;{\rm and} \;\;\; H^1(\mathcal F_u) \cong  H^1(K_C-D) \oplus H^1(K_C-p)$$for any 
$u \in \mathcal W_1$, we will prove the existence of such an injective map \eqref{eq:injectivity1} for the splitting bundle 
$\mathcal F_0 := (K_C-D) \oplus (K_C-p) \in  \mathcal W_1$.

Concerning the domain of the map $\psi^{\vee}_C$, i.e. $H^0(\mathcal F_0) \otimes H^0(\omega_C \otimes \mathcal F_0^*)$, 
as in \cite[Proof of Prop.\;3.9]{CFK} one has 
\[\begin{array}{ccl}
H^0(\mathcal F_0) \otimes H^0(\omega_C \otimes \mathcal F_0^*) & \cong &  \left(H^0(K_C-D) \otimes H^0(D)\right)  \oplus \left(H^0(K_C-D) \otimes H^0(p)\right) \oplus\\
& & \left(H^0(K_C-p) \otimes H^0(D)\right) \oplus \left(H^0(K_C-p) \otimes H^0(p)\right). 
\end{array}
\]On the other hand, since 
$$\det (\mathcal F_0) = 2 K_C - p - D \;\; {\rm and} \;\; 
Sym^2(\mathcal F_0^*)= (p + D - 2 K_C) \oplus (2p - 2 K_C) \oplus (2D - 2K_C),$$one has  
$$\omega_C \otimes Sym^2(\color{black} \mathcal F_0\color{black}^*) \otimes \det (\color{black} \mathcal F_0\color{black}) \cong K_C \oplus (K_C + p - D) \oplus (K_C+D-p).$$Therefore, concerning the target of the map $\psi^{\vee}_C$, one has: 
\[\begin{array}{ccl}
{\small H^0\left(\omega_C \otimes Sym^2(\color{black} \mathcal F_0\color{black}^*) \otimes \det (\color{black} \mathcal F_0\color{black}) \right) \oplus H^0\left(\omega_C^{\otimes 2}\right)} & \cong &  
H^0(K_C) \oplus H^0(K_C + p - D) \oplus\\
& & H^0(K_C+D-p) \oplus H^0(2K_C). 
\end{array}
\] \color{black} By the above decomposition of $H^0(\mathcal F_0) \otimes H^0(\omega_C \otimes \mathcal F_0^*)$ and of 
$H^0\left(\omega_C \otimes Sym^2(\mathcal F_0^*) \otimes \det (\mathcal F_0) \right) \oplus H^0\left(\omega_C^{\otimes 2}\right)$, one considers the following natural maps:
\begin{eqnarray*}
\mu_{0,D}: & H^0(D)\otimes H^0(K_C-D)\to H^0(K_C),\\
 \mu_{p,K_C-D}:& H^0(p) \otimes H^0(K_C-D)\to H^0(K_C-D+p)\\
 \mu_{D,K_C-p}: & H^0(D) \otimes H^0(K_C-p)\to H^0(K_C+D-p)\\
 \mu_{0, p}: & H^0(p)\otimes H^0(K_C-p) \to H^0(K_C),  
 \end{eqnarray*}(which are simply defined by multiplication of global sections of line bundles and are all injective as 
$h^0(D) = h^0(p) =1$) and the following natural injection: 
$$\iota : H^0(K_C) \hookrightarrow H^0(2K_C),$$which is induced by any choice of an effective divisor in $|K_C|$. Looking at the Chern classes of the involved line bundles, \color{black} one naturally defines 
$$\psi_C^{\vee} := \mu_{0, D} \oplus \mu_{p,K_C-D} \oplus \mu_{D,K_C-p} \oplus (\iota \circ \mu_{0, p})$$\color{black} which is therefore \color{black} injective. Moreover, it is compatible with the dual maps 
 $H^1(\rho_*(\gamma_S))^{\vee}$ and $H^1(\rho_*(\tau_S))^{\vee}$ as $\mathcal F_0$ splits.  
 
The previous argument shows $(ii)$, completing the proof.  
\end{proof}

\color{black}
To conclude the proof of Theorem \ref{thm:Hilb} (i), the fact that $\mathcal Y$ fills--up a unique component, say $\mathcal H_{\rm reg}$, 
with all the properties mentioned therein, it suffices to observe that 
$$ 7g-7 + k_2(k_2-2) = \dim\; \mathcal Y \leq \dim \; T_{[S]} (\mathcal H_{d,g,k_2-1}) = h^0(S,\N_{S / \Pp^{k_2-1}})$$ and to use Proposition \ref{prop:Normalreg} (i). 
The fact that $\mathcal H_{\rm reg}$ is a {\em regular} component of $\mathcal H_{d,g,k_2-1}$ follows from the fact that 
$\chi(S, \N_{S / \Pp^{k_2-1}}) = h^0(S,\N_{S / \Pp^{k_2-1}})$ as in \eqref{eq:tgS3bis}, i.e. $\mathcal H_{\rm reg}$ is reduced and of expected dimension.


\end{document}